\documentclass{amsart}

\usepackage{amsmath,amssymb,amsthm,amsfonts}
\usepackage{setspace}

\newtheorem{theorem}{Theorem}

\newtheorem{prop}{Proposition}[section]
\newtheorem{lemma}[prop]{Lemma}

\newtheorem{definition}[prop]{Definition}
\newtheorem{remark}[prop]{Remark}
\newtheorem{notation}[prop]{Notation}
\newtheorem{claim}[prop]{Claim}

\def\A{\mathcal{A}} \def\B{\mathcal{B}} \def\F{\mathcal{F}} 
\def\G{\mathcal{G}} \def\I{\mathcal{I}} \def\K{\mathcal{K}}

\def\C{\mathcal{C}} 

\def\II{\mathfrak{I}} \def\LL{\mathfrak{L}} \def\PP{\mathfrak{P}} 
\def\RR{\mathfrak{R}} \def\ZZ{\mathfrak{Z}}

\def\E{\mathbb{E}}  \def\P{\mathbb{P}}
\def\Q{\mathbb{Q}} \def\R{\mathbb{R}} \def\Z{\mathbb{Z}}

\def\ruule{\mathcal{R}}

\def\un{\mathbf{1}} \def\dd{\mathrm d}
\def\ee{\mathrm{e}} \def\eps{\varepsilon}

\def\Amb{\mathrm{Amb}} \def\pred{\mathrm{Preced}} \def\perf{\mathrm{Perf}}

\def\Tmes{\mathbb{T}}
\def\t{\vartheta} \def\projo{\pi}
\def\shif{\sigma} \def\di{\mathrm{Influ}}
\def\local{\ell} 

\renewcommand{\le}{\leqslant} \renewcommand{\ge}{\geqslant}

\def\ctwa{coupling time with ambiguities}

\def\puri{R} \def\pyri{Y}
\def\sensi{\mathrm{sen}} \def\insen{\mathrm{ins}}
\def\tautaux{\delta}

\begin{document}


\author{Jean B\'{e}rard, Didier Piau}

\address[Jean B\'{e}rard]{\noindent Universit\'e de Lyon ;
Universit\'e Lyon 1 ;
Institut Camille Jordan CNRS UMR 5208 ;
43, boulevard du 11 novembre 1918,
F-69622 Villeurbanne Cedex; France.
 \newline e-mail: \rm
  \texttt{Jean.Berard@univ-lyon1.fr}}

\address[Didier Piau]{\noindent Institut Fourier - UMR 5582,
  Universit\'e Joseph Fourier Grenoble 1, 100 rue des Maths, BP 74,
  38402 Saint Martin d'H\`eres, France.
  \newline e-mail: \rm \texttt{Didier.Piau@ujf-grenoble.fr}}

\date{\today}

\title[Coupling times with ambiguities for particle systems]{Coupling times 
with ambiguities for particle systems and applications to 
context-dependent DNA substitution models}

\keywords{Interacting particle systems, Coupling, Perturbations, Stochastic models of 
nucleotide substitutions}

\subjclass[2000]{60J25, 60K35, 92D20}

\begin{abstract}
We define a notion of \ctwa\ for  
interacting particle systems, and show how this can be used to prove ergodicity and 
to bound the convergence time to equilibrium and the decay of correlations 
at equilibrium. A motivation is to provide simple conditions which ensure that 
perturbed particle systems share some properties of the underlying unperturbed system.
We apply these results to context-dependent substitution models recently introduced by 
molecular biologists as descriptions of DNA evolution processes. These models take into account
the influence of the neighboring bases on the substitution probabilities at a site of the DNA sequence,
 as opposed to most usual substitution models which assume that sites evolve independently of each other.
\end{abstract}

\maketitle


\hrulefill
\setcounter{tocdepth}{1}
\tableofcontents
\vskip -5ex
\hrulefill

\section{Introduction and motivations}

This paper is devoted to interacting particle systems on the integer line $\Z$ with
finite state space $S$, whose dynamics is characterized by a finite
list $\RR$ of stochastic transition rules. We
now give an informal description of the dynamics that we consider for these systems, and
we postpone a proper mathematical definition to section~\ref{s:def}.

\subsection{Construction of interacting particle systems dynamics}

We begin with some vocabulary.
A \textit{state\/} $s$ is an element of $S$, a
\textit{site\/} $x$ is an element of $\Z$, a \textit{configuration\/}
$\xi:=(\xi(x))_{x\in\Z}$ is an element of $S^{\Z}$.  A \textit{rule\/}
$\ruule:=(c,r)$ is based on a context $c$ and characterized by a rate
$r$. A \textit{context\/} is a triple $c:=(A,\local,s)$, where $A$ is
a finite subset of $\Z$, $\local$ is a subset of $S^A$, $s$ is a
state, and $r$ is a \textit{rate\/}, that is, a non-negative real
number.

We say that a configuration $\xi$ and a context $c=(A,\local,s)$, or
any rule $\ruule=(c,r)$ based on $c$, are \textit{compatible at site
  $x$\/} if $A$ is empty, or if $A$ is not empty and $\xi(x+A)$ belongs
  to $\local$, where $\xi(x+A)$ is the element of $S^A$ defined as
$$\xi(x+A) := (\xi(x+y))_{y \in A}.$$
 
The interacting particle system is a Markov process $(X_{t})_t$ on
$S^{\Z}$ whose dynamics is characterized by a given finite list $\RR:=(\ruule_{i})_{i \in \II}$ of stochastic transition rules,
as follows: for any time $t$,
if a rule $\ruule_{i}=(c_{i},r_{i})$ in $\RR$ with $c_{i}=(A_{i},\local_{i},s_{i})$ 
is compatible with $X_{t}$ at site $x$, then $X_{t+\dd
  t}(x)=s_{i}$ with probability $r_{i}\dd t+o(\dd t)$, independently of
every other rule in $\RR$, compatible with
$X_{t}$ at site $x$ or elsewhere.

A classical way to give a more explicit construction of such particle
systems uses the so-called graphical representation (see for instance~\cite{Lig2} 
page 142 for a discussion in the context of voter
models).  This amounts to a stochastic flow based on Poisson
processes: given a time $t$ and an initial condition $\xi$ in $S^{\Z}$
imposed at time $t$, the Poisson processes determine the state of the
particle system at every time greater than $t$.  
Once again informally, to every
site $x$ and rule $\ruule_{i}=(c_{i},r_{i})$ in $\RR$ corresponds
a homogenous Poisson process $\Psi(x,i)$ on the real line $\R$ with
rate $r_{i}$, and the points of $\Psi(x,i)$ are the random times at which
the rule $\ruule_{i}$ is applied to the state at site $x$.
Specifically, for every rule $\ruule_{i}=(c_{i},r_{i})$ in $\RR$ with context
$c_{i}=(A_{i},\local_{i},s_{i})$, if $t$ belongs to
$\Psi(x,i)$ and if $\ruule_{i}$ and $X_{t^{-}}$ are
compatible at site $x$, then $X_{t}(x)=s$; otherwise,
$X_{t}(x)=X_{t-}(x)$.  See section~\ref{s:def} for a proper
definition.

\subsection{Coupling times}

Within this framework, various notions of coupling times can be
defined. In this paper, an \textit{ordinary coupling time\/} is an almost surely
finite random variable $T$ with negative values, measurable with
respect to the family $(\Psi(x,i))_{(x,i)\in\Z\times\II}$ of Poisson
processes, and such that, for every time $u<T$, if the dynamics
starts at time $u$, the state of site $x=0$ at time $t=0^-$ is
the same for every initial condition at time $u$. This definition
corresponds to a coupling from the past, as opposed to the usual notion
of forward coupling.

As soon as such coupling times exist, the particle
system is ergodic. Furthermore, estimates on the tail of $T$ yield estimates
on the rate of convergence to equilibrium, and additional assumptions on
the coupling time yield estimates on the decay of correlations.
Consider now the set of points 
$$ 
\mathcal{T}:=\bigcup_{(x,i)} 
\left( \Psi(x,i) \cap [T,0[ \right) \times \{  x \},
$$ 
where the union runs over every $x$ in $\Z$ and $i$ in $\II$.  A point
in $\mathcal{T}$ corresponds to a transition that may or may not be
performed between the times $t=T$ and $t=0^-$, depending on the
initial condition at time $v<T$.  When, for a given $(u,x)$ in
$\mathcal{T}$, there indeed exists $v< T$ and two distinct initial
conditions at time $v$ such that, for one of these initial conditions,
the transition proposed by $(u,x)$ is performed, while it is not
performed when the other initial condition is used, we say that an
ambiguity arises at $(u,x)$.  By the definition of an ordinary
coupling time, one sees that, for each time in $\mathcal{T}$, either
there is no ambiguity associated with it, or there is an ambiguity
that has no influence on the state of site $x=0$ at time $t=0^-$.

We can now define, once again informally, the notion of \ctwa.  This
is a pair $(H,T)$, where $T$ is a random variable with negative
values, measurable with respect to the family
$(\Psi(x,i))_{(x,i)\in\Z\times\II}$ of Poisson processes and
$H$ is a finite random subset of the set $\mathcal{T}$ defined above,
enjoying the stopping property, and such that the following property
holds: for any two initial conditions at time $u<T$ such that the
ambiguities associated with the elements of $H$ are resolved in the
same way (that is, a transition corresponding to an element of $H$ is
either performed for both initial conditions, or not performed for
both initial conditions), the state of site $x=0$ at time $t=0^-$ is
the same for both initial conditions.
    
One sees that, if $(H,T)$ is a \ctwa, $T$
may or may not be an ordinary coupling time.  
However, the only ambiguities that may prevent $T$ from being an ordinary
coupling time are those associated to the points in $H$. As a consequence,  
in the degenerate case when $H$ is empty, $T$ is indeed an ordinary coupling
time.
  
Informally, our main result is that, if the random set $H$ contains
few enough points on average (we call \textit{subcritical\/} any 
\ctwa\ enjoying this property), it is possible to build an
ordinary coupling time  from $(H,T)$, thus proving ergodicity of the
particle system.  
Moreover, more specific estimates and assumptions about the set
$H$ provide estimates on this ordinary coupling time, that are suitable to
study the rate of convergence to equilibrium of the particle system
and the decay of its correlations.
    
The construction of an ordinary coupling time from a subcritical
\ctwa\, is described in section~\ref{s:proof}. The principle of this
construction is to apply iteratively coupling times with ambiguities,
looking further and further into the past, until every ambiguity is
eventually resolved.
    
\subsection{Perturbed particle systems}    
    
We now describe how these results allow to study some perturbed
particle systems.
We assume that the list of transition rules is of the form
$\RR=( \RR_{i}  )_{i \in \II^{o} \cup \II^{p}}$, where $\II^{o}$ and $\II^p$ are disjoint sets, 
the family $\RR^{o}:=(\RR_{i})_{i \in \II^{o}}$ 
being the family of so-called 
non-perturbative rules, while $\RR^{p}:=(\RR_{i})_{i \in \II^{p}}$ is the family of so-called 
perturbative rules. 

We call the interacting particle system based on the whole family of
rules $\RR$  the perturbed system and the system based on the family of
non-perturbative rules $\RR^o$ the unperturbed 
system.  

A general problem about perturbations of particle systems is to relate
the properties of the perturbed system such as ergodicity, speed of
convergence to equilibrium or decay of correlations at equilibrium, to
those of the unperturbed system, when the transition rates attached to
the perturbative rules are small enough.  In this context,
we wish to mention two
results, one on the negative side and one on the positive side:
\begin{enumerate}
\item Small
perturbations of ergodic particle systems may not be ergodic.
For a well-known example, consider the two-dimensional Ising model. Its dynamics is ergodic at
the critical inverse temperature $\beta_{c}$ and not ergodic at 
any inverse temperature $\beta >\beta_{c}$, see~\cite{Lig} (page 204 and
Theorem 2.16 on page 195) for instance.
\item  Small
perturbations of particle systems whose coordinates evolve
independently are ergodic, see~\cite{Lig} (Theorem 4.1 on page 31) for instance.
\end{enumerate}
 
Depending on the assumptions one makes about the unperturbed system,
and on the perturbations one considers, one can use various methods 
to deal with this problem. For an example of the use of $L^{2}$
techniques and spectral gap estimates, see~\cite{KomOll}. For an example of
the use of cluster-expansion estimates, see~\cite{MaeNet}. For examples
involving Lyapounov function techniques, in the slightly different
context of perturbations of Markov chains,
see~\cite{AltAvrNun,GlyMey,RobRosSch,ShaStu}.

The approach of this paper is based on coupling. The basic idea is
that, in some situations, it should be possible to rely on the
coupling properties of the unperturbed system to devise a \ctwa\
$(H,T)$ for the perturbed system, with the property that, when the
rates associated with perturbative rules are small enough, the \ctwa\
is subcritical. We do not provide an abstract formulation of this
idea, but, as an illustration, we give two concrete examples in
section~\ref{s:appli}, in the context of stochastic models of
nucleotide substitution in molecular evolution, recently studied in~\cite{BerGouPia}.

\subsection{Organization of the paper}

Section~\ref{s:def} contains a formal definition of the interacting particle systems 
studied in this paper, their construction by means of Poisson processes, 
the definition of the notion of \ctwa, and some notations.   
Section~\ref{s:stat} contains the main results, whose proofs are in
section~\ref{s:proof}.
Section~\ref{s:appli} applies these theoretical results to a concrete case, namely a class of
stochastic models of context-dependent nucleotide substitution, recently introduced by molecular biologists,
and whose study was our initial motivation for the results
in this paper.

%
\section{Formal setting}\label{s:def}

\subsection{Preliminary definitions and notations}

In this paper, particle systems are continuous-time Markov processes
on $S^{\Z}$, where $S$ denotes a finite set. Sites $x$ are elements of
$\Z$, states $s$ are elements of $S$ and configurations 
$\xi=(\xi(x))_{x\in\Z}$
are elements of
$S^{\Z}$.  The space $\C:=\C( [0,+\infty[, S^{\Z})$ is the space of
c\`adl\`ag functions on $[0,+\infty[$ with values in $S^{\Z}$.  For
every nonnegative time $t$, $X_{t}:\C\to S^{\Z}$ is the canonical coordinate
map on $\C$, hence $X_{t}(\omega):=\omega(t)$ for every $\omega$ in
$\C$.   The space
$\C$ is endowed with the cylindrical
$\sigma$-algebra $\sigma((X_{t})_{t \ge 0})$.  For every nonempty subset $B$
of $\Z$, $\projo_B:S^{\Z}\to S^B$ is the canonical
projection defined by
$$
\projo_B(\xi):=(\xi(x))_{x \in B}.
$$
For every site $x$, $\projo_{x}:=\projo_{\{x\}}$.  For every site $y$,
$\t_{y}:S^{\Z}\to S^{\Z}$ is the canonical translation of $S^{\Z}$
defined by
$$
\t_{y}(\xi):=(\xi(x+y))_{x \in \Z}.
$$
For every configuration $\xi$, site $x$ and state $s$, the
configuration
$\xi^{x,s}$ is defined by $\xi^{x,s}(x):=s$, and
$\xi^{x,s}(y):=\xi(y)$ for every site $y\neq x$.

Finally, $\R_{+}:=[0,+\infty[$, $\R_{-}:=]-\infty,0]$, and $C(S^{\Z})$
is the space of the functions $F$ defined on $S^{\Z}$ such that the following
series converges:
$$ 
\sum_{x\in\Z} \sup\{\,| F( \xi^{x,s})-F(\xi)|\,;\,\xi\in S^{\Z},\,s\in S\}. 
$$

\subsection{Specifications by transition rules}

Recall that one is given a finite list of transition rules
$$
\RR=(\ruule_{i})_{i\in\II},\qquad \ruule_{i}=(c_{i},r_{i}),
$$
indexed by a finite set $\II$.
For every $i$ in $\II$, the rate $r_{i}$ of the rule $\ruule_{i}$ 
is a nonnegative real number and its context $c_{i}=(A_{i},\local_{i},s_{i})$
is characterized by a finite subset
$A_{i}$ of $\Z$, a subset $\local_{i}$ of $S^{A_{i}}$ and a state $s_{i}$ in $S$. 

The list of rules $\RR$ yields a definition of the dynamics of the process
through its infinitesimal generator $\LL$, as follows: for every
function $F$ in $C(S^{\Z})$ and every configuration $\xi$,
$$
(\LL F)(\xi) := \sum_{(x,i)}
r_{i}\,\un\{\projo_{A_{i}}(\t_{x} \xi) \in \local_{i}\}\,
(F(\xi^{x,s_{i}}) -F(\xi)),
$$
where the sum enumerates every site $x$ in $\Z$ and rule index $i$ in $\II$.
Here and below, we adopt the convention that
$\un\{\projo_{\emptyset}(\xi) \in \emptyset \}=1$ for every $\xi$ in $S^{\Z}$.
 
The fact that the generator $\LL$
indeed defines a Feller Markov process $(X_{t})_{t\ge0}$ 
on $S^{\Z}$ is a standard result, see~\cite{Lig} for instance. 
For every configuration $\xi$, $\Q^{\xi}$ denotes the probability measure on $\C$
which corresponds to the initial condition $X_{0}=\xi$.

Distinct families of stochastic transition rules may lead to the same
infinitesimal generator $\LL$.  As a consequence, the probability
measures $\Q^{\xi}$ do not uniquely determine $\RR$ in general, and 
several families of rules are compatible with
the same Markov process. However, distinct families of rules do lead
to distinct versions of the construction presented in the next
section, so the coupling properties of this construction that are of
essential use in this paper, may differ substantially from one family
of rules to another, even when the corresponding infinitesimal
generators are the same.

%
\subsection{Dynamics based on Poisson processes}

The infinitesimal generator $\LL$ defined above is the usual way to
specify a dynamics from a finite collection of transition rules.
However, the coupling times that we consider in this paper are
formulated in terms of another construction, the so-called graphical
construction of the dynamics, see~\cite{Lig}, which is based on a
family of Poisson processes corresponding to transition times.

We now describe this construction in details.
Since we will be interested in coupling from the past, we only have to define the dynamics 
of the process for the ``past'' times $t \le 0$.

\subsubsection{Poisson processes}

The individual sample space for the Poisson processes is the set 
$$ \Omega_{0}:= \{ (t_k)_{k\ge1}\subset\R\,;\,\forall k\ge1,\, t_{k+1}<t_{k}<0,
\,\lim_{n \to +\infty} t_{n} =   -\infty  \}.
$$
We equip $\Omega_{0}$ with the $\sigma$-algebra $\F_{0}$  induced by the product Borel $\sigma$-algebra 
on the space of real valued sequences.
For every $k \ge 1$, the coordinate map $T_{k}:\Omega_{0}\to\R$ is defined by 
$$
T_{k}((t_{n})_{n\ge1}):= t_{k}.
$$
For every site $x$ in $\Z$ and rule index $i$ in $\II$,
 $\P_{x}^{i}$ is the probability measure on $(\Omega_{0},\F_{0})$ such that the sequence 
$(T_{k})_{k\ge1}$ is a Poisson process on $\R_{-}$ with rate $r_{i}$.

To define the dynamics of interest, we introduce a family of processes
on the sample probability space 
$$
(\Omega, \F,\P):=\bigotimes_{(x,i)}(\Omega_{0}, \F_{0},\P_{x}^{i}),
$$ 
where the $\otimes$ product enumerates every site $x$ in $\Z$ and every 
rule index $i$ in $\II$. 
 For every $x$ in $\Z$, $i$ in $\II$, $k\ge1$, and 
$\psi$ 
in $\Omega$,  one defines
$$
\Psi(x,i,k)(\psi) := T_{k}(\psi(x,i)),\quad
\Psi(x,i)(\psi):=\psi(x,i),
$$
and
$$
\Psi(x)(\psi):=(\psi(x,i))_{i\in \II},\quad
\Psi(\psi):=\psi.
$$
With these notations, 
$\Psi(x,i)=(\Psi(x,i,k))_{k \ge1}.
$ 
In the sequel, $\Psi(x,i)$ denotes also the random set  
$\displaystyle\bigcup_{k \ge1} \{\Psi(x,i,k)\}$ 
and the context should
make clear which one of these two notations is in use. 
The same convention applies to $\Psi(x)$ and $\Psi$.
Finally, $\F(x,i)$ is the sub-$\sigma$-algebra of $\F$ generated by $\Psi(x,i)$.

\begin{remark}\label{r.not}
In our context, it is necessary to use an indexation of the various random variables $\Psi$ and of related
 quantities by rule indices $i$ in $\II$ instead of an indexation by the rules $\ruule$ in $\RR$ themselves. To see why, 
consider the case when
two rules $\ruule_{i}$ and $\ruule_{j}$ with $i\neq j$ in $\II$ are described by the same 
contexts 
$c_{i}=c_{j}$ and the same rates $r_{i}=r_{j}$. Then $\ruule_{i}=\ruule_{j}$
but we want to consider the addition of their two effects, which could also be described by 
the single rule with context $c_{i}$ and rate $2r_{i}$.
\end{remark}

\subsubsection{Flows}
 
Let $\Omega_{1}\subset\Omega$ denote the event that $\Psi(x,i,k)\ne\Psi(x',i',k')$
for every $(x,i,k)\ne(x',i',k')$.
Then $\P(\Omega_{1})=1$, that is, almost surely, for each time $t\le0$,
$t$ belongs 
to exactly one set
$\Psi(x,i,k)$ or $t$ belongs to none of them. 

\begin{definition}[Direct influences]
  The direct influence process is the random process $\di$ defined on
  $\R_{-} \times \Z$ as follows.  Let $t\le0$ and $x$ in $\Z$.  If $t$
  belongs to a unique set $\Psi(x,i,k)$, let
$$
\di(t) :=   \{  t    \}  \times (x+A_{i}).
$$ 
Otherwise,
let $\di(t):= \{ (t,x) \}$. 
Conversely, for every site $x$, let 
$$
\di^{-1}(t,x):=t.
$$
\end{definition}

When $A_{i}$ is empty, this definition implies that $\di(t)$ is empty.

\begin{definition}[Preceding times]
For every times $u<t\le0$ and site $x$, 
the preceding time $\pred_{u}(t,x)$ at $x$ after $u$ and before $t$ is the random 
variable with values in $\R_{-} \times \Z$ defined by
$$
\pred_{u}(t,x):=  \left( \sup\,   ]u,t[ \cap \Psi(x),x \right),
$$
with the convention that $\pred_{u}(t,x):=(u,x)$ if $]u,t[\cap \Psi(x)$ 
is empty. 
\end{definition}

\begin{definition}[Multilevel influences]
For every site $x$ and times $u<t\le0$, we define
 inductively a sequence $(\di_{k}(u,t,x))_{k\ge0}$ of random sets,
 called the influences of site $x$ after $u$ and before $t$ at level $k$, 
as follows.

\begin{itemize}
\item
For $k=0$,
let $\di_{0}(u,t,x):= \di(t,x)$; 
\item
For every $k\ge0$, 
$$\di_{k+1}(u,t,x):=\di_{k}(u,t,x) \cup \di( \pred_{u}(\di_{k}(u,t,x) )). 
$$
\end{itemize}

Additionally, the complete influence of site $x$ after $u$ and before $t$ is
$$
\di_{\infty}(u,t,x):= \bigcup_{k=0}^{+\infty} \di_{k}(u,t,x).
$$
\end{definition}

Let $\Omega_{2}\subset\Omega$ denote the event that
$\di_{\infty}(u,t,x)$ is a finite set for every site $x$ and couple
$(u,t)$ of times such that $u<t \le 0$. Then $\P(\Omega_{2})=1$ (see~\cite{Lig2} 
page 142, for instance).

We define a random flow $\Phi$ on $S^{\Z} \times \R \times \R \times
\Z$, such that, for every $u \le t \le 0$, $\Phi(\xi,u,t,x)$ is the
$x$-coordinate of the configuration at time $t$ which one obtains by
applying the moves described by $\Psi$ to the configuration $\xi$ at
time $u$.

The definition of $\Phi$ is based on the following recursive procedure.

Assume first that $\Omega_{1}\cap\Omega_{2}$ holds. 
Fix a site $x$ and times $u\le t \le 0$.
If $u=t$, let $\Phi(\xi,u,t,x) := \xi(x)$. 
If $u<t$, consider first the case 
where $t$ is not in $\Psi(x)$. Then $\di(t,x)=\{(t,x)\}$, and we use a recursive
call to the definition of $\Phi$ by letting 
$$
\Phi(\xi,u,t,x) := \Phi(\xi, \pred_{u}(t,x)). 
$$
Otherwise, $t = \Psi(x,i,k)$ for exactly one rule index $i$ in $\II$ and one index $k\ge1$.
Consider then the set $\pred_{u}(\di(t,x)) $. 
If this set reduces to $\{ (u,x) \}$, then, for any
$(y,t)$ in $\di(t,x)$,  no rule applies at site $y$ between the times 
$u$ and $t$. Then, let 
$\chi(y) :=  \xi(y)$.
Otherwise, $\pred_{u}(\di(t,x)) $ is not reduced to $\{(u,x)\}$.
For every element $(y,t)$ of $\di(t,x)$, we use a recursive call 
to the definition of $\Phi$ and let 
$$
\chi(y) :=  \Phi(\xi,u, \pred_{u}(t,y)).
$$
Then, if $\chi(x+A_{i})$ belongs to $\local_{i}$ (remember that this is
automatically the case when $A_{i}$ is empty), let $\Phi(\xi,u,t,x) := s_{i}$, and
say that $\Psi(x,i,k)$ is performed when one starts from configuration
$\xi$ at time $u$.  Otherwise, let $\Phi(\xi,u,t,x) := \xi(x)$ and say
that $\Psi(x,i,k)$ is not performed when one starts from configuration
$\xi$ at time $u$.  For the sake of definiteness, if
$\Omega_{1}\cap\Omega_{2}$ does not hold, let
$\Phi(\xi,u,t,x):=\xi(x)$, and say that $\Psi(x,i,k)$ is not
performed, whatever the value of $(x,i,k)$ is.  This ends the
description of the construction of $\Phi$.  

The fact that, on
$\Omega_{1} \cap \Omega_{2}$, $\di_{\infty}(u,t,x)$ is a finite set,
guarantees that the above procedure involves only a finite number of
recursive calls to the definition of $\Phi$ and leads to a consistent
definition of $\Phi$.  Moreover, one can check that the fact that
$\Psi(x,i,k)$ is performed or not does not depend on the value of $t$,
but only on $\xi$, $u$, and, of course, $\Psi$.  The proof of the
proposition below is standard.
\begin{prop}[Flow properties]\label{p.flot}
The flow $\Phi$ enjoys the following properties.
\begin{itemize}
\item For every times $u<v<t\le 0$ and site $x$, 
$$\Phi(\Phi(\xi,u,v,\cdot), v,t, x) = \Phi(\xi,u,t,x).$$
\item For every time $u\le 0$, the distribution of the c\`adl\`ag random process 
$$
(\Phi(\xi,u,u+t,\cdot))_{0\le t\le-u}
$$ 
is the distribution of $(X_{t})_{0\le t\le-u}$ with respect to $\Q^{\xi}$.
\end{itemize}
\end{prop}

A motivation to give the details of the construction of $\Phi$ was to be able to define 
the following random variable.

\begin{definition}[Performance indicator]
The performance indicator of rank $k\ge1$ for the rule index $i$ at site $x$, 
starting from configuration $\xi$ at time $u$, is
 $$
\perf(\xi,u,x,i,k) :=  \un\{\Psi(x,i,k)\,
\mbox{is performed when starting from $\xi$ at time $u$}\}.
$$
\end{definition}

\subsubsection{Measurability and shifts}\label{ss.ms}

For every time $t \le 0$ and rule index $i$ in $\II$, let $K_{t}(x,i)$ 
denote the random variable on $(\Omega, \F,\P)$ defined by
$$
K_{t}(x,i) := \max \{  k\ge1\, ; \,  \Psi(x,i,k)  \ge t \}\cup\{0\}.
$$
For every time $t \le 0$ and site $y$,
 the space-time-shift $\shif_{t,y}$ is defined on $\Omega$ by 
$$ [\shif_{t,y}(\Psi)](x,i,k) :=  \Psi \left( x+y,i,k+K_{t}(x+y,i)\right) - t.$$
Then $\P$ is invariant with respect to every $\shif_{t,y}$ and 
$$\un_{\Omega_{1}\cap\Omega_{2}} \le \un_{\Omega_{1}\cap\Omega_{2}} 
\circ \shif_{t,y}. 
$$
The behavior of the flow under the action of the shift is described by our next lemma, whose proof
is left to the reader.

\begin{lemma}\label{l:flotshift}
On $\Omega_{1}\cap\Omega_{2}$, 
for 
every sites $x$ and $y$ and times $u \le v \le0$ and $t\le 0$,
$$
\Phi(\xi, u, v, x) \circ \shif_{t,y}     =        \Phi(\xi, u+t,v+t,x+y),
$$
and, for every rule index $i$ in $\II$ and index $k\ge1$,
$$
\perf(\xi,u, x, i,   k- K_{t}(x,i)    ) \circ \shif_{t,y}  =   
\perf(\xi,u+t,x+y,i,k). 
 $$
\end{lemma}

Let $\displaystyle\Gamma :=  \bigcup_{n \ge 0} \{   n \} \times\R^{n}$
and $\F^{\Gamma}$ the $\sigma$-algebra  on $\Gamma$ generated by 
the sets $\{ n \} \times B$, for every $n \ge 0$ and every Borel subset $B$ of $\R^n$.

\begin{definition}
For every time $t\le0$, let
$\F^{+}(t)$ denote the sub-$\sigma$-algebra of $\F$ generated by the family of maps 
$\theta_{t}^{+}(x,i):\Omega\to\Gamma$, for every site $x$ in $\Z$ and rule index $i$ 
in $\II$, 
defined by
$$ \theta_t^{+}(x,i)  := 
\left(  K_{t}(x,i) \,;\,\{\Psi(x,i,k)\,;\,1\le k\le K_{t}(x,i)\}\right).
$$
More generally, if $U$ is a random variable defined on $(\Omega,\F)$ with values in $\R_{-}$,
$\F^{+}(U)$ denotes the  sub-$\sigma$-algebra of $\F$ generated by the maps 
$\theta_{U}^{+}(x,i)$.
\end{definition}

One can view $\F^{+}(t)$ as the $\sigma$-algebra of the events posterior to the time $t$.

%
\subsection{Coupling times with ambiguities}\label{ss.dctwa}

\begin{definition}\label{d:inde}
Let  $H=\{H(x,i,k)\,;\,x\in\Z,\,i\in\II,\,k\ge1\}$ denote a family of random variables, defined on $(\Omega,\F,\P)$ and with values in
 $\{ 0 , 1 \}$. Then 
 $\Psi H$ is the subset of $\R\times\Z$ defined as
$$
\Psi H:=\{   (\Psi(x,i,k), x)\,;\,x\in\Z,\, i\in\II,\,k\ge1,\,H(x,i,k)=1\}.
$$
Likewise, 
$$
\perf(\xi,u,H):=\left\{\perf(\xi,u,x,i,k)  H(x,i,k)\,;\,x \in\Z,\, 
i \in \II,\,k \ge 1\right\}.
$$
Let $\K$ denote the product $\sigma$-algebra on 
$\{ 0 , 1 \}^{ \Z \times \II \times \{1, 2, \ldots \}}$.
\end{definition}

\begin{definition}[Coupling time with ambiguities]\label{d:ctwa}
The pair $(H,T)$ is a \ctwa\ if 
$H=\{H(x,i,k)\,; \,x\in\Z,\,i \in \II, \, k \ge 1\}$
is a $\{0,1\}$-valued
process and  $T$ is a random variable defined on 
$(\Omega, \F, \P)$, such that the
following holds.

 \begin{enumerate}
\item\label{ctwa1} The random variable $T$ belongs to $]-\infty,0[$, $\P$ almost surely.
\item\label{ctwa2} At most a finite number of the random variables $H(x,i,k)$ are not zero,
$\P$ almost surely.
\item\label{ctwa3} For every site $x$, rule index $i$ and index $k \ge 1$, 
if $\Psi(x,i,k) < T$, then $H(x,i,k)=0$, $\P$ almost surely. 
\item\label{ctwa4} For every site $x$, rule index $i$ and index $k \ge 1$,  
$H(x,i,k)$
is measurable with respect to $\F^{+}(\Psi(x,i,k))$.
\item\label{ctwa5}  For every time $t < T$ and configurations
$ \xi$ and $\xi'$, if
$\perf(\xi,t,H)$ and $\perf(\xi',t,H)$ are equal, then
$ \Phi(\xi,t,0^{-},0)$ and $\Phi(\xi',t, 0^{-},0)$ are equal, $\P$ almost surely.
\end{enumerate}
\end{definition}

\begin{remark}\label{r:fonction}
A consequence of definition~\ref{d:ctwa} and of the flow property of $\Phi$ described by proposition~\ref{p.flot} 
is that, if $(H,T)$ is a 
\ctwa, there exists
a map $\Theta:\Omega \times \{ 0 , 1 \}^{ \Z \times \RR \times \{1, 2, \ldots \}}\to S$ such that,
$\P$ almost surely,  for every $t<T$,  
$$
\Phi(\xi,t,0^{-},0)  =  \Theta( \Psi,  \perf(\xi,t,H)).
$$
Note that $\Theta$ does not depend on $t$.
\end{remark}

\begin{definition}[Width of coupling times with ambiguities]\label{d:bwi}
The width of a coupling time with ambiguities $(H,T)$ is bounded by 
the couple $(a_{-},a_{+})$ of nonnegative integers if the following holds.  

\begin{enumerate}
\item 
The random process $H=\{H(x,i,k)\,; \,x\in\Z,\,i\in \II, \, k \ge 1\}$ 
is measurable with respect to $\sigma(\F(x,i)\,;  \,  -a_{-} \le x \le  a_{+} , \,   i \in \II   )$.
\item 
$\di(\Psi H) \subset [-a_{-},a_{+}] \times \R_{-}$, $\P$ almost surely. 
\item\label{i.mes}
The map $\Theta$ in remark~\ref{r:fonction}  can be chosen to be measurable 
with respect to the $\sigma$-algebra
$ \sigma \left( \F(x,i)\,;  \,  -a_{-} \le x \le  a_{+},  \,  i\in \II \right) \otimes \K$.
\end{enumerate}
\end{definition}

\begin{definition}[Growth parameter]\label{d.meanvmap}
The growth parameter of a \ctwa\ $(H,T)$ is 
$$ 
m(H,T) := \E[\# \di(\Psi H)].
$$
If $m(H,T)<1$, we say that the coupling time with ambiguities $(H,T)$ is subcritical.
\end{definition}

Note that
$$
\# \di(\Psi H) = \sum_{(x,i,k)}(\#A_{i})\, H(x,i,k).
$$

\begin{definition}[Laplace transforms of coupling time with ambiguities]\label{d.lapvmap}
For every real number $\lambda$,  
introduce $\Lambda_{T}(\lambda) :=\E[\ee^{-\lambda T}]$
and
$$
\Lambda_{H}(\lambda):=  \E \left(  \sum_{(t,y)} 
 \ee^{-\lambda t}\un\{(t,y) \in \di(\Psi H)\} \right).
$$
\end{definition}

Recall that $T<0$ almost surely and that $t<0$ for every $(t,y)$ in $\di(\Psi H)$.
Note that $\Lambda_{T}(0)=1$, $\Lambda_{H}(0)=m(H,T)$, and
$$
\sum_{(t,y)}  \ee^{-\lambda t}\un\{(t,y) \in \di(\Psi H)\} 
=
\sum_{(x,i,k)} (\# A_{i})\,\ee^{-\lambda \Psi(x,i,k)}  H(x,i,k).
$$ 

\section{Statement of the main results}\label{s:stat}

We are now able to state the main results of this paper.

\begin{theorem}[Ergodicity]\label{t:ergod}
If there exists a subcritical \ctwa, 
the particle system is ergodic. That is, there exists a
  unique invariant probability distribution $\mu$, and, for every initial configuration $\xi$, 
 $X^{\xi}_t$ converges in
  distribution to $\mu$ as $t$ goes to infinity.
\end{theorem}

Theorems~\ref{t:cvexp} and \ref{t:corrdecay} below provide non-asymptotic results.
Theorem~\ref{t:exp2} is a consequence of theorem~\ref{t:cvexp}.

\begin{theorem}[Explicit bound]\label{t:cvexp}
Assume that there exists a subcritical \ctwa\ $(H,T)$ and let $\mu$ 
denote the unique invariant probability distribution of the particle system. 
For every 
configuration $\xi$, finite subset of sites $B\subset \Z$,  and time $t \ge 0$,
the distance in total variation between the distributions $\projo_{B}( X^{\xi}_t)$ 
and $\projo_{B}(\mu)$ is at most
$$
\# B \times \inf_{n\ge 0} \left( m(H,T)^n + \sum_{k=1}^{n}
  \inf_{\lambda \ge 0} \Lambda_{H}(\lambda)^{k}
  \Lambda_{T}(\lambda)\ee^{ - \lambda t} \right) .
$$ 
\end{theorem}

\begin{theorem}[Exponential rate of convergence]\label{t:exp2}
Assume that there exists a subcritical \ctwa\ $(H,T)$ with finite width such that 
$T$ is exponentially integrable. Then, with respect to the total variation distance,
for every initial configuration $\xi$,
the finite marginals of $X^{\xi}_t$ converge exponentially fast 
to the finite marginals of the invariant
distribution.
\end{theorem}

\begin{theorem}[Decay of correlations]\label{t:corrdecay}
Assume that there exists a subcritical \ctwa\ with growth parameter $m<1$
and finite width bounded by  $(a_{-},a_{+})$.
Let $\mu$ denote the unique invariant probability distribution of the particle system. 
For every real number $z$, let $\kappa(z):=|z|/(a_{+}+a_{-})$.
\\
For every sites $x$ and $y$ in $\Z$,
the distance in total variation between 
$\projo_{\{x,y\}}(\mu)$
and $\projo_{x}(\mu) \otimes \projo_{y}(\mu)
=
\projo_{0}(\mu) \otimes \projo_{0}(\mu)$ is at most 
$2m^{\kappa(y-x)-1}$.
\\
Let $x$ denote a positive integer, 
$B:=\Z\cap[x,+\infty)$ and $C:=\Z\cap(-\infty,0]$. 
The distance in total variation between 
$\projo_{B\cup C}(\mu)$
and $\projo_{B}(\mu) \otimes \projo_{C}(\mu)$ is at most 
$$
m^{\kappa(x)-1} 
\left(\frac{1}{1-  m^{1/a_{-}} }    +     
\frac{1}{1-  m^{1/a_{+}}}\right).
$$
As a consequence, the same bound applies to every subsets $B$ and $C$ of $\Z$ such that
$\min B\ge x+\max C$.
\end{theorem}



%
\section{Proof of the main results}\label{s:proof}

In section~\ref{ss.ap}, we define a crucial tool for our proofs, namely the notion of ambiguity processes.
In section~\ref{ss:cap}, we explain how to control these.
This enables us to prove theorem~\ref{t:ergod} in section~\ref{ss:enda},
theorem~\ref{t:cvexp} in section~\ref{ss:endb}, 
and theorem~\ref{t:corrdecay} in section~\ref{ss:proofexp2}.
Section~\ref{ss:prep} is a preparation to the proof of theorem~\ref{t:exp2},
given in section~\ref{ss:endd}. Finally, section~\ref{ss:mes} settles some measurability issues.

\subsection{Ambiguity processes}\label{ss.ap}

For every sites $x$ and $y$, time $t\le0$, rule index $i$ in $\II$ and index 
$k \ge 1$, let
$$
H(t,x)(y,i,k)  :=   (H \circ \shif_{t,x})(y-x,i,k-K_{t}(y,i)).
$$
For every site $x$ and time $t\le0$, let
$$
H(x,t):=\{ H(x,t)(y,i,k)\,  ; \, y \in \Z, \, i \in \II, \, k \ge 1\}. 
$$
Similarly, let 
$$
T(x,t) :=  (T \circ \shif_{t,x}) +  t.
$$
In words, the pair $(H(t,x),T(t,x))$ corresponds to the translation of the \ctwa\ 
$(H,T)$ from site $0$ and time $0^{-}$ to site $x$ and time $t^{-}$.

\begin{definition}
 For every site $x$ and time $t \le 0$, let $A(t,x)$ denote the event that,  
for every time $u< T(x,t)$ and configurations $ \xi$ and $\xi'$ 
such that $\perf(\xi,u, H(t,x))$ and $\perf(\xi',u, H(t,x))$ coincide, 
$\Phi(\xi,u,t^{-},x)$ and $\Phi(\xi',u,t^{-},x)$ coincide.

Let $\Omega_{3}$ be the event 
$$\Omega_{3}:=A(0,0) \cap \bigcap_{(x ,y,i,k)} A(\Psi(y,i,k),x).
$$
\end{definition}

\begin{lemma}\label{l.no}
\textbf{(1)} For every site $x$ and time $t \le 0$, $\P[A(t,x)]=1$.
\\
\textbf{(2)} For every sites $x$ and $y$, rule index $i$ and index $k\ge1$, 
$$
\P[A(\Psi(y,i,k),x)] = 1.
$$   
\end{lemma}

\begin{proof}[Proof of lemma~\ref{l.no}]
Part (1) is a simple consequence of 
lemma~\ref{l:flotshift} in section~\ref{ss.ms}, and of the fact that 
$\P$ is invariant under the action of $\shif_{t,x}$.
We omit the details of the proof.

As regards part (2), by our definition~\ref{d:ctwa} above, there exists a set $B$ in $\F$ such that 
$B \subset A(0,0)$ and $\P[B]=1$.
With our definitions, for every $x$ and $t$, 
$$
\Omega_{1} \cap\Omega_{2} 
\cap \{  \psi \in \Omega\,; \,  \shif_{t,x}(\psi) \in B  \}  \subset  A(t,x).
$$
Hence, we only need to prove that $\P[ \shif_{\tau,x}(\Psi) \in B  ]   = 1$, 
where $\tau:=\Psi(y,i,k)$.

By standard properties of Poisson processes,  $ \shif_{\tau,x}(\Psi)$ 
is independent from $\tau$ and has the same distribution as $\Psi$.
As a consequence, letting $\Tmes$ denote the distribution of $\tau$, 
the distribution of $( \tau,  \shif_{\tau,x}(\Psi)    )$ on $(-\infty,0) \times \Omega $ 
equipped with the product $\sigma$-algebra $\B((-\infty,0)) \otimes \F$, 
is equal to the product measure $\Tmes \otimes \P  $.
By Fubini theorem, 
$$
\P[ \shif_{x,\tau}(\Psi) \in B  ]   =   \int_{-\infty}^{0}  \P[ \Psi \in B    ]  
\,\dd\Tmes(t) = 1.
$$
This concludes the proof of lemma~\ref{l.no}. 
\end{proof}

A direct consequence of lemma~\ref{l.no} above is the following proposition.

\begin{prop} 
$\P[\Omega_{3}]=1$.
\end{prop}

\begin{definition}[Ambiguity processes]\label{d.amb}
The ambiguity process at site $x$ is a sequence $(\Amb_n(x))_{n\ge0}$
of
random subsets of $\Z \times \R_{-}$, defined recursively as follows.

\begin{itemize}
\item Initialization:
$\Amb_0(x) := \{ (0,x) \}$.
\item Induction:
$\displaystyle
\Amb_{n+1}(x) := \bigcup_{ (t,y) \in \Amb_n(x)  } \di(\Psi H(t,y)).
$
\end{itemize}

This defines a nondecreasing random sequence of sets $(\Amb_n(x))_{n \ge 0}$.
Let 
$$
\Amb(x) := \bigcup_{n \ge 0} \Amb_n(x).
$$
\end{definition}

By construction,  $\Amb(x)$ is a finite set if and only 
$\Amb_n(x)=\Amb_{n+1}(x)$ for some index $n\ge0$. For such an index $n$, 
$\Amb_{k}(x)=\Amb_{n}(x)$ for every $k\ge n$, 
whence $\Amb_{n}(x)= \Amb(x)$.  

We wish to prove that for subcritical coupling times with ambiguities,
the set $\Amb(x)$ is almost surely finite. 

\begin{definition}[Coupling time at a site]\label{d:mpt}
The coupling time $T^{*}_{x}$ at site  $x$ is
$$
T^{*}_{x} := \inf\{  T(t,y) \,;  (t,y) \in \Amb(x) \}.
$$
\end{definition}

Observe that, if $\Amb_{n}(x)=\Amb_{n+1}(x)$, then 
$$
T^{*}_{x} := \inf   \{ T(t,y) \,;\, (t,y) \in \Amb_{n}(x) \}.
$$

\begin{definition}[Influence in the ambiguity process]
Let $(t,y)$ and $(u,z)$ denote elements of $\Amb(x)$. Say that $(u,z)$ is 
influenced by $(t,y)$ if $(u,z)$ belongs to the set $\di(\Psi H(t,y))$.
\end{definition}

Note that an element of $\Amb(x)$ may be influenced by several elements of  $\Amb(x)$.

We define inductively a sequence $(Z_{n}(x))_{n \ge 0}$
of random subsets of $\Amb(x)$ as follows.
For $n=0$, let
$Z_{0}(x):=\{ (0,x) \}$. For every $n\ge0$,
$Z_{n+1}(x)$ denotes the set (possibly empty) of the elements influenced by elements of $Z_{n}(x)$.
Hence, 
$$
\Amb(x) = \displaystyle\bigcup_{n \ge 0} Z_{n}(x).
$$ 
Moreover, if $Z_{n}(x)$ is empty for a given $n \ge 1$, then $Z_{k}(x)$ is empty for every 
$k \ge n$ as well, 
and in that case,
$$
\Amb(x) = \displaystyle\bigcup_{0 \le k  \le n-1} Z_{k}(x). 
$$
\begin{definition}[Locked sites]
For every site $x$ and times $u<t \le 0$, say that $(t^{-},x)$ is locked by time $u$ 
if for every time $v < u$ and configurations $\xi$ and $\xi'$, 
  $$\Phi(\xi,v,t^{-},x) = \Phi(\xi',v,t^{-},x).$$
  \end{definition}

\begin{lemma}\label{l:verrouill}
On $\Omega_{1}\cap\Omega_{2}$, if   $(t^{-},x)$ is locked by time $u<t$, then  
for every times $v<u$ and $v'< u$, and 
configurations $\xi$ and $\xi'$,   
 $$ \Phi(\xi,v,t^{-},x) = \Phi(\xi',v',t^{-},x)$$
\end{lemma}

The proof of lemma~\ref{l:verrouill} is a consequence of the flow property of $\Phi$ 
in proposition~\ref{p.flot}
and we omit it.

We now give a definition concerning ambiguities.

\begin{definition}[Resolution of ambiguities]
One says that the ambiguity associated to $(\Psi(x,i,k),x)$ is resolved by time $u$ if $u<\Psi(x,i,k)$ 
and if, for every configurations $\xi$ and $\xi'$,
$$
\perf(\xi,u,x,i,k) = \perf(\xi',u,x,i,k).
$$
\end{definition}

\begin{prop}\label{p:verrouill}
If $\Amb(x)$ is finite, then $(0^{-},x)$ is almost surely locked by time $t$, 
for every $t < T^{*}_{x}$.
\end{prop}

\begin{proof}[Proof of proposition~\ref{p:verrouill}]
Assume throughout the proof that $\Omega_{1}\cap\Omega_{2}\cap\Omega_{3}$ holds,
since this event has probability one. 

Consider an index $n$ such that       
$\Amb_{x}(n)=\Amb_{x}(n+1)$, and let $(t_{k},x_{k})_{1\le k\le r}$ denote an enumeration of  
the set $(\Amb_{x}(n)) \cap ]-\infty,0]$ such that 
$$
t_1 < t_2 < \cdots < t_r = 0.
$$
We wish to prove by induction that, for every $ 1 \le  q \le r$, with full probability, 
the following property $(P_q)$ holds: 

\begin{quote} $(P_{q})$
For every $1 \le j \le q$, and every $(t_{j},y)$ in $\di(t_{j}, x_{j})$, 
$(t_{j}^{-},y)$ is locked by time $T^{*}_{x}$.
\end{quote}

Assume first that $q=1$ and consider $(t_1,z)$ in $\di(t_1)$.  

By definition, $\di(\Psi H(t_1,z))\subset\Amb_{x}(n+1)$  
 since $(t_1,z)$ belongs to $\Amb_{x}(n)$.
But every
element $(w,u)$ in $\di(\Psi H(t_1,z))$ is such that $u< t_1$, by definition of $H$ and $\di$. 
On the other hand, $u>t_1$
by the definition of $t_1$. This is a contradiction, hence $\di(\Psi H(t_1,z))$ is empty. 
Using the definition of a \ctwa\
and the fact that, by definition, $T^*_x \le T(t_1,z)$, we deduce that $(t_1^{-},z)$ is, 
with full probability, locked by time $T^{*}_{x}$. This proves $(P_{1})$.

Assume now that $(P_q)$ hold for some $1\le q \le r-1$, and 
consider an element $(t_{q+1},z)$ of $\di(t_{q+1},x_{q+1})$.    Observe   that
$H(t_{q+1},z)\subset\{  (t_1,x_{1}), \ldots,   (t_q,x_{q}) \}$.
As a consequence, according to $(P_q)$, with full 
probability, the ambiguities associated with the elements of $H(t_{q+1},z)$ are resolved by time $T^{*}_{x}$.
Thus, for every configurations $\xi$ and $\xi'$ and time $t<T^{*}_{x}$, 
$$
\perf(t,\xi, H(t_{q+1},z)) =\perf(t,\xi', H(t_{q+1},z)).
$$ 
Using the fact that by definition $T^*_x \le T(t_{q+1},z)$, and the definition of a 
\ctwa, one sees 
that $(t_{q+1}^{-},z)$ is locked by time $T^{*}_{x}$. Hence  the ambiguity 
associated with $(t_{q+1},x_{q+1)}$ is resolved by time $T^{*}_{x}$, and $(P_{q+1})$ 
holds.

The proof of proposition~\ref{p:verrouill} is complete.
\end{proof}

\begin{remark}
The reader might have noticed that property~(\ref{ctwa4}) in definition~\ref{d:ctwa} in section~\ref{ss.dctwa}
is not used in the proof of proposition~\ref{p:verrouill} above. 
However, this property plays a crucial role in the estimates on the ambiguity process presented in the next
section.
\end{remark}

%
\subsection{Controlling ambiguity processes}\label{ss:cap}

The goal of this section is to prove the preliminary estimates of lemma~\ref{l:branch} below.
For every $\lambda$ and nonnegative integer $n$, let 
$$
C(n,\lambda) := \sum_{(t,y) \in Z_{n}(0)} \ee^{-\lambda t},
$$
and
$$
D(n,\lambda):=\sum_{(t,y) \in Z_{n}(0)} \ee^{-\lambda T(t,y)}.
$$

\begin{lemma}\label{l:branch}
For every $\lambda$ and nonnegative integer $n$,
$$
\E[C(n,\lambda)]\le \Lambda_{H}(\lambda)^{n},
$$
and
$$
\E[D(n,\lambda)] =
\Lambda_{T}(\lambda)    \E( C(n,\lambda) )
\le\Lambda_{T}(\lambda) \Lambda_{H}(\lambda)^{n}.
$$
\end{lemma}

\begin{proof}[Proof of lemma~\ref{l:branch}]
The proof of the first assertion is by induction on $n$. For $n=0$, $C(0,\lambda)=1$ hence
the result is obvious. 
Assume that the result holds for a given $n \ge 0$. 
Every element in $Z_{n+1}(x)$ is influenced by at least one element in $Z_{n}(x)$, hence
$C(n+1,\lambda)  \le C_{0}(n+1,\lambda)$, with
$$
C_{0}(n+1,\lambda):= \sum_{(t,y) \in Z_{n}(0)}  \sum_{(u,z)}    
\ee^{-\lambda u}\un\{(u,z) \in \di(\Psi H(t,y))\}.
$$ 
Hence,
$$
C_{0}(n+1,\lambda)  =\sum_{(t,y) \in Z_{n}(0)} \ee^{-\lambda t}    
\sum_{(u,z)}    \ee^{-\lambda (u-t)}\un\{(u,z) \in \di(\Psi H(t,y))\}.
$$
Let 
$$
C_{1}(n,\lambda,x,y,i,k):= \un\{(y,\Psi(x,i,k)) \in Z_{n}(0)\}   
\,\ee^{-\lambda    \Psi(x,i,k)  },
$$
and 
 $$
C_{2}(\lambda,x,y,i,k):= \sum_{(u,z)  }    
\ee^{-\lambda (u-  \Psi(x,i,k))}\,
\un\{(u,z) \in \di(\Psi H(\Psi(x,i,k), y))\}.
$$
The last expression of $C_{0}(n+1,\lambda)$ can be rewritten as
$$  
C_{0}(n+1,\lambda)
=
\sum_{(x,i,k)} \sum_{y \in x+A_{i}}  
C_{1}(n,\lambda,x,y,i,k) C_{2}(\lambda,x,y,i,k).
$$
Taking expectations on both sides,  
 $$
\E [C_{0}(n+1,\lambda)] =\sum_{(x,i,k)} \sum_{y \in x+A_{i}}   
\E \left[    C_{1}(n,\lambda,x,y,i,k)  C_{2}(\lambda,x,y,i,k) \right].
$$
According to lemmas~\ref{l:mesurable2} and~\ref{l:mesurable3} 
in section~\ref{ss:mes} below,
every $ C_{1}(n,\lambda,x,y,i,k)   $ is measurable with respect to 
$\F^+(  \Psi(x,i,k))$, while the conditional distribution of 
every $C_{2}(\lambda,x,y,i,k)$ with respect to $\F^+(  \Psi(x,i,k))$
is the same as the (unconditional) distribution of 
$$
\sum_{(u,z)}    
\ee^{ -\lambda u  }\un\{(u,z) \in \di(\Psi H)\}.
$$ 
As a consequence, 
$$
\E    \left[ C_{1}(n,\lambda,x,y,i,k)  C_{2}(\lambda,x,y,i,k) \right]
=   
 \Lambda_{H}(\lambda) \,\E[C_{1}(n,\lambda,x,y,i,k)].
$$
In turn, this implies that
 $$\E[ C_{0}(n+1,\lambda)]
=
\Lambda_{H}(\lambda)
 \E \left( \sum_{(x,i,k)} \sum_{y \in x+A_{i}}     
C_{1}(n,\lambda,x,y,i,k)   \right).
$$
 It remains to notice that
$$ 
\sum_{(x,i,k)} \sum_{y \in x+A_{i}}     
C_{1}(n,\lambda,x,y,i,k)  = C(n,\lambda),
$$
to see that the induction on $n$ is complete.
The proof of the first assertion of lemma~\ref{l:branch} is complete. 

As regards the second assertion, fix an integer $n \ge 0$, and define
$$
C_{3}(\lambda,x,y,i,k):=  \ee^{-\lambda(T(\Psi(x,i,k))-\Psi(x,i,k), y)}.
$$
Using the functional $C_{1}$ defined in the proof of the first assertion, one sees that
$$
D(n,\lambda) = \sum_{(x,i,k)} \sum_{y \in x+A_{i}}    
C_{1}(n,\lambda,x,y,i,k) C_{3}(\lambda,x,y,i,k).
$$
As a consequence, 
$$
\E[D(n,\lambda)]
=    \sum_{(x,i,k)} \sum_{y \in x+A_{i}}  
 \E[C_{1}(n,\lambda,x,y,i,k) C_{3}(\lambda,x,y,i,k)].
$$
As in the proof of the first assertion, 
we observe that, for every fixed $(x,i,k,y)$, the random variable
$ C_{1}(n,\lambda,x,y,i,k)$ is measurable with respect to 
$\F^+(  \Psi(x,i,k))$, while the conditional distribution of  
$  T(y,\Psi(x,i,k))  - \Psi(x,i,k)    $  with respect to $\F^+(  \Psi(x,i,k))$,
is the same as the (unconditional) distribution of $ T$.
Using the last displayed identity in the proof of the first assertion once again, 
the result follows.

The proof of lemma~\ref{l:branch} is complete. 
\end{proof}

\subsection{End of the proof of theorem~\ref{t:ergod}}\label{ss:enda} 
Lemma~\ref{l:branch} with $\lambda=0$ shows that, for every $n \ge 0$ and every site $x$, 
$\E[\# Z_{n}(x)]\le \Lambda_{H}(0)^{n}=m(H,T)^{n}$. Markov inequality yields
 $\P[ Z_{n}(x) \neq \emptyset ] \le m(H,T)^{n}$. 
In particular, there exists $\P$ almost surely an integer $n$ such that $Z_{n}(x)$ is empty.
For such an integer $n$, 
$$\Amb(x) =\displaystyle \bigcup_{0 \le k \le n-1} Z_{k}(x),
$$ 
hence 
$\Amb(x)$ is  finite  with full probability.

Fix a finite subset $B \subset \Z$, and let $T^{*}_{B}:=\min\{T^{*}_{x}\,;\,x\in B\}$.
By proposition~\ref{p:verrouill} in section~\ref{ss.ap}, for every
 time $t<T^{*}_{x}$ and configurations $\xi$ and $\xi'$,
$$ \projo_{B}( \Phi(\xi,t,0^{-},x) )= \projo_{B}(\Phi(\xi',t,0^{-},x)).$$
The probability that time $0$ belongs to $\Psi$ is $0$, hence, almost surely, 
for every
 time $t<T^{*}_{x}$ and configurations $\xi$ and $\xi'$,
$$ 
\projo_{B}( \Phi(\xi,t,0,x) )= \projo_{B}(\Phi(\xi',t,0,x)).$$
As a consequence, for every positive time $t$ and
configurations $\xi$ and $\xi'$,
the distance in total variation between 
$\projo_{B}( X^{\xi}_t)$ and $\projo_{B}( X^{\xi'}_t)$ is at most $\P[T^{*}_{B} < -t ]$.
Since every $T^{*}_{x}$ is almost surely finite and $B$ is finite,
$T^{*}_{B}$ is almost surely finite, hence
this distance goes to $0$ when $t$ goes to infinity.
 
Now, a generic compactness argument shows that, as a consequence of $S$ being finite, there exists at least one invariant distribution $\mu$ for the particle system
(see for instance~\cite{Lig} chapter 1, proposition 1.8). 
Consider now a random configuration
$\xi'$ with distribution $\mu$, and an arbitrary configuration $\xi$. For every $t \ge 0$, 
the distribution of  $ X^{\xi'}_t$ is $\mu$, and the previous estimates
then show that, as $t$ goes to infinity, $X^{\xi}_{t}$ converges in distribution to $\mu$ as $t$ goes to infinity. 
This ends the proof of theorem~\ref{t:ergod}.

\subsection{End of the proof of theorem~\ref{t:cvexp}}\label{ss:endb}
By the union bound,  for every negative $t$,
 $$
\P[T^{*}_{B} \le t ] \le \sum_{x \in B} \P[T^{*}_{x} \le t ]  
= 
(\# B)\,\P[T^{*}_{0} \le t ].
$$
Fix an index $n\ge1$. On the event that $Z_{n}(0)$ is empty,  
$$
T^{*}_{0} = \min\{T(u,x)\,;\,0\le k\le n-1,\,(u,x) \in Z_{k}(0)\}.
$$
As a consequence, by the union bound,
$$
\P[T^{*}_{0} \le t ] \le  m(H,T)^{n} +  \sum_{k=0}^{n-1} 
\P \left[ \min_{(u,x) \in Z_{k}(0)} T(u,x) \le t\right].
$$
Fix $\lambda \ge 0$. Then  
$$  
\exp( -\lambda\min_{(u,x) \in Z_{k}(0)} T(u,x)   )   \le  D(k,\lambda).
$$ 
Hence, Markov inequality yields  
$$
\P \left[  \min_{(u,x) \in Z_{k}(0)} T(u,x) \le t \right]  
\le \ee^{\lambda t} \E[D(k,\lambda)]
\le  \ee^{\lambda t }\Lambda_{H}(\lambda)^{k} \Lambda_{T}(\lambda),
$$
where lemma~\ref{l:branch} provides the last inequality.
Taking the infimum with respect to $\lambda\ge0$ in this inequality, 
separately for each $0\le k\le n-1$, yields the conclusion
of theorem~\ref{t:cvexp}.

\subsection{Proof of theorem~\ref{t:exp2}}\label{ss:proofexp2}

By hypothesis, $\Lambda_{T}(\lambda)$ is finite for some positive values of $\lambda$ and the width
of $(H,T)$ is finite, hence our next lemma implies the theorem.

\begin{lemma}\label{l:ht}
Assume that the width of $(H,T)$ is finite
and that $\Lambda_{T}(\lambda)$ is finite for a positive $\lambda$. 
Then $\Lambda_{H}(\lambda')$ is finite for every $\lambda'<\lambda$.
\end{lemma}

\begin{proof}[Proof of lemma~\ref{l:ht}]
Replacing every $H(x,i,k)$ such that $-a_{-}\le x\le a_{+}$
and $\Psi(x,i,k)\ge T$ by $1$ in the expectation which 
defines $\Lambda_{H}(\lambda')$ yields 
$$
\Lambda_{H}(\lambda')\le\sum_{-a_{-}\le x\le a_{+}}\sum_{i}
\E[\Lambda_{H}(x,i,\lambda')],
$$
with
$$
\Lambda_{H}(x,i,\lambda'):=\sum_{t\in\Psi(x,i)}\ee^{-\lambda' t}
\un\{t>T\}.
$$
Fix a positive real number $\lambda''$.
Then, $\ee^{-\lambda' t}\un\{t>T\}\le\ee^{-\lambda'' T+(\lambda''-\lambda')t}$,
hence, for every site $x$ and rule index $i$,
$$
\Lambda_{H}(x,i,\lambda')\le\ee^{-\lambda'' T}
\sum_{t\in\Psi(x,i)}\ee^{(\lambda''-\lambda') t}.
$$
Fix $p>1$ and $q>1$ such that $1/p+1/q=1$. By Minkowski inequality,
$$
\E[\Lambda_{H}(x,i,\lambda')]\le\E[\ee^{-p\lambda'' T}]^{1/p}
\sum_{k\ge1}\E[\ee^{q(\lambda''-\lambda')\Psi(x,i,k)}]^{1/q}.
$$
Since each $-\Psi(x,i,k)$ is the sum of $k$ i.i.d.\ 
exponential random variables, the last sum is
the sum of a geometric series, with ratio 
$\E[\ee^{q(\lambda''-\lambda')\Psi(x,i,1)}]^{1/q}$. Assume that 
$\lambda''>\lambda'$. The ratio is less than $1$, 
hence the sum over $k\ge1$ converges.

Summing this over every site $-a_{-}\le x\le a_{+}$ and rule index $i$ yields a finite upper bound of
$\Lambda_{H}(\lambda')$ as soon as one can find $\lambda''>\lambda'$ and $(p,q)$ 
such that $1/p+1/q=1$ and $\E[\ee^{-p\lambda'' T}]$ is finite.
If $\Lambda_{T}(\lambda)=\E[\ee^{-\lambda T}]$ is finite, this is possible for every $\lambda'<\lambda$,
hence the proof of lemma~\ref{l:ht} is complete.
\end{proof}

\subsection{Preparation to the proof of theorem~\ref{t:corrdecay}}\label{ss:prep}

\begin{definition}\label{d:nstar}
For every site $x$, let $N^{*}(x)$ denote the, almost surely finite, smallest integer $n$ such that $Z_{n}(x)$ is empty,
and let 
$\G_{x}$ denote the $\sigma$-algebra 
$$
\G_{x}:=\sigma(\F(y,i)\,;  \, x+ a_{-}  N^{*}(x)    \le y \le x+ a_{+} N^{*}(x),\,  
i \in \II   ).
$$
\end{definition}

\begin{lemma}\label{66}
For every site $x$, there exists a random variable $F_{x}$, with values in $S$, 
which is measurable with respect to $\G_{x}$ and such that,
$\P$ almost surely and for every configuration $\xi$, 
$$
\lim_{t \to -\infty}  \Phi(\xi,t, 0, x)  =  F_{x}.
$$
\end{lemma}

\begin{proof}[Proof of lemma~\ref{66}]
Assume throughout the proof that $\Omega_{1}\cap\Omega_{2}\cap\Omega_{3}$ holds,
since this event has probability one. 
Let $(t_{k},x_{k})_{1\le k\le r}$ denote an enumeration of  
the set $\di^{-1}(\Amb_{x}(N^{*}(x))) \cap ]-\infty,0]$ such that $t_1 < t_2 < \cdots < t_r = 0$.

For every $1\le m \le r-1$, define $(x_{m}, i_{m}, k_{m})$
by the relation
$$
t_{m}=: \Psi(x_{m}, i_{m}, k_{m}).
$$ 
We define by induction a sequence $(W_{m})_{0 \le i \le r-1}$ such that, 
 for every $m$,
$$
W_{m} \in  \{ 0 , 1 \}^{ \Z \times \II \times \{1, 2, \ldots \}}.
$$
First, let $W_{0}(x,i,k):=0$ for every $x$, $i$ and $k$.
Now, let $ 1 \le m \le r-1$.
For every $q \le m-1$,
let $W_{m}(x_{q},i_{q},k_{q}):=  W_{m-1}(x_{q},i_{q},k_{q})$.
For $q=m$, let
$$
W_{m}(x_{m},i_{m},k_{m}):= 
\un \left\{\left[\Theta(\shif_{t_{m},y}(\Psi)), W_{m-1} H(t_{m},x_{m})  
\,;\,y\in x_{m}+A_{m}\right]    \in   \local_{m} \right\}.
$$ 
Finally, let $W_{m}(x,i,k) := 0$ when $(x,i,k)$ is not one of the triples 
$(x_{j},i_{j},k_{j})$ for $1\le j\le m$. 

As in the proof of proposition~\ref{p:verrouill}, we have that, for every $q \le r-1$,  $\P$ almost surely,  
for every time $t<T_{x}^{*}$ and for every $(t_{q+1},y)$ in $\di(t_{q+1},x_{q+1})$, 
$\Psi H(t_{q+1},y)$ is a subset of $\{ ( t_1,x_{1}),\ldots, (t_{q},x_{q}) \}$,
and that
$$
\Phi(\xi,t,t_{q+1}^{-},y) = \Theta(  \perf(\xi, t, H(t_{q+1},y) ) , 
\shif_{t_{q+1},y}(\Psi)).
$$ 
Hence
$\perf(\xi, t,  H(t_{q},y) ) = W_{q-1} H(t_{q},y) $ for every $1 \le q \le r$
and $\P$ almost surely. 
Let 
$$
F_{x}:= \Theta( W_{r-1} H(x,0), \shif_{0,x}(\Psi)).
$$
One sees that, 
$\P$ almost surely and  for every time $t<T_{x}^{*}$, 
$$
\Phi(\xi,t,0,x) = \Phi(\xi,t,0^{-},x) =
F_{x}.
$$
Finally, the measurability properties of $F_{x}$ follow from assumption~(\ref{i.mes}) 
in definition~\ref{d:bwi}.
The proof of lemma~\ref{66} is complete.
 \end{proof}

\subsection{Proof of theorem~\ref{t:corrdecay}}\label{ss:endd}

Assume that $\min B =n+\max C$ with $n\ge 1$, and introduce the real number
$$
z:=  \max C +  n\frac{a_{+}}{a_{+}+a_{-}}=\max C+\kappa(n) a_{+}. 
$$  
Let $\Psi_{C}$ and $\Psi_{B}$ denote two families 
of Poisson processes indexed by $\Z$ and such that the following properties hold.

\begin{itemize}
  \item  $\Psi_{C}(x)=\Psi(x)$ for every $x < z$ and  $\Psi_{B}(x)=\Psi(x)$ for every $x > z$.
 \item $(\Psi_{C}(x))_{x \ge  z}$ has the same distribution as  $(\Psi(x))_{x \ge z}$ 
but is independent 
from $(\Psi,\Psi_{B})$.
\item $(\Psi_{B}(x))_{x \le  z}$ has the same distribution as  $(\Psi(x))_{x \le z}$ but 
is independent 
from $(\Psi,\Psi_{C})$.
\end{itemize} 

Hence $\Psi_{B}$ and $\Psi_{C}$ are independent.  Recall lemma~\ref{66}
and definition~\ref{d:nstar} in section~\ref{ss:prep}
and let 
$$
N_{C} = \sup \{  x + a_{+} N^{*}(x) ; \, x \in C\},
$$
and
$$
N_{B} = \inf \{  x -a_{-} N^{*}(x) ; \, x \in B\}.
$$ 
Consider the event 
$$
E:=\{N_{C} < z<N_{B}\}.
$$
On $E$,
$\left((F_{x}(\Psi))_{x\in C},(F_{x}(\Psi))_{x\in B}\right)$ is distributed like
$\left((F_{x}(\Psi_{C}))_{x\in C},(F_{x}(\Psi_{B}))_{x\in B}\right)$.
Since $\Psi_{C}$ and $\Psi_{B}$ are independent,  
$(F_{x}(\Psi))_{x\in C}$ and $(F_{x}(\Psi))_{x\in B}$ 
are independent on $E$. Hence 
the distance in total variation we want to estimate is at most $\P[\Omega\setminus E]$.

Note that
$\Omega\setminus E=\{N_{B} \le z\}\cup\{N_{C} \ge z\}$.
As regards $N_{B}$,
$$\displaystyle\{N_{B} \le z\}=\bigcup_{x \in B}   
\{ x - a_{-} N^{*}(x)   \le z \}, 
$$
hence
$$
\P[N_{B} \le z]\le \sum_{x\ge \min B} \P[x -  a_{-} N^{*}(x)   \le z].
$$
For each $x\ge\min B$, 
$\{x - a_{-} N^{*}(x)   \le z \}= \{N^{*}(x) \ge  \kappa(n) + k/a_{-}\}$, provided $k$
is the nonnegative integer
$k:=x-\min B$.
Furthermore, for every real number $v$, $\P[N^{*}(x) \ge v]\le m^{v-1}$ 
with $m:=m(H,T)<1$. 
This yields
$$
\P[N_{B} \le z]\le \sum_{k\ge 0} m^{\kappa(n) -1+ k/a_{-}}=
\frac{m^{\kappa(n)-1}}{1-m^{1/a_{-}}}.
$$ 
The same argument, applied to $\P[N_{C} \ge z]$, 
and the union bound, yield the result. 

The case when $C=\{x\}$ and $B=\{y\}$ with $y>x$ is similar, except that 
one can replace the geometric series involved above by their first term, hence
the tighter bounds.

The proof of theorem~\ref{t:corrdecay} is complete.

%
\subsection{Measurability properties}\label{ss:mes}

Let $\mu(H)$ denote the random counting measure associated to $\Psi H$, that is, 
$$
\mu(H) := \sum_{(x,i,k)} \un\{H(x,i,k)=1\}\, \delta_{\Psi(x,i,k)}. 
$$

\begin{lemma}\label{l:mesurable2}
Fix sites $x$ and $y$, a rule index $i$, an index $k\ge1$, 
and a nonnegative measurable function $F:\R\to\R_{+}$.
Then the random variable $\displaystyle\int F\dd\mu(H \circ \shif_{\Psi(x,i,k),y})$
 is independent from $\F^+( \Psi(x,i,k))$ and
 has the same distribution as $\displaystyle\int F\dd\mu(H)$.
\end{lemma}

\begin{proof}[Proof of lemma~\ref{l:mesurable2}]
The independence property is a consequence of the independence of $\shif_{\Psi(x,i,k),y}$ and 
$\F^+( \Psi(x,i,k))$.
The equidistribution property is a consequence of the invariance of $\P$ with respect to $\shif_{\Psi(x,i,k),y}$.
\end{proof}

\begin{lemma}\label{l:mesurable3}
For every integer $n\ge0$, the event $ \{\Psi(x,i,k) \in Z_{n}(x)\} $ 
is measurable with respect to 
$\F^+(  \Psi(x,i,k))$.
\end{lemma}

\begin{proof}[Proof of lemma~\ref{l:mesurable3}]
We proceed by induction on $n\ge0$. The case of $n=0$ is included in the definition of a \ctwa.
Let $n\ge1$.
By definition, $\{ \Psi(x,i,k) \in Z_{n}(x) \}$ is the union of the events
$$
E(x',i',k',y):=\left\{\Psi(x',i',k')\in Z_{n-1}(x),\,\Psi(x,i,k)\in H(y, \Psi(x',i',k'))\right\},
$$
over every site $x'$, rule index $i'$, index $k'$ and site $y$ in $x'+A_{i'}$.

In the rest of this proof, we use $\tau$ and $\tau'$ as shorthands for $\tau:=\Psi(x,i,k)$ 
and 
$\tau':=\Psi(x',i',k')$, respectively.

Lemma~\ref{l:mesurable3} follows from claims~\ref{c01} and \ref{c02} below.

\begin{claim}\label{c01} For every event $G$ in $\F^+( \tau')$, 
$G \cap \{ \tau  < \tau' \}$ belongs to $\F^+( \tau)$.
\end{claim}

\begin{claim}\label{c02} 
The event
$\{ \tau  \in H(\tau',y) \}  $ 
belongs to  $\F^+( \tau)$. 
\end{claim}

Indeed, by the induction hypothesis,  
$\{  \tau'\in Z_{n-1}(x) \}$ belongs to $\F^+( \tau')$. 
Hence,  using claim~\ref{c01}, we see that 
$\{  \tau' \in Z_{n-1}(x) \} \cap \{ \tau  < \tau'  \}$
belongs to  $\F^+( \tau)$.
On the other hand, by claim~\ref{c02}, 
we see that $\{ \tau  \in H(\tau',y)\}$ belongs to  $\F^+( \tau)$ too.
Since,  by definition, $\{ \tau  \in H(\tau',y)\}\subset\{ \tau  < \tau' \}$,
we finally deduce that any event $E(x',i',k',y)$ defined above
belongs to 
 $\F^+( \tau)$.
This ends the proof of lemma~\ref{l:mesurable3}.
\end{proof}

\begin{proof}[Proof of claim~\ref{c01}]
It is easily checked, using the fact that $\{ \tau  < \tau' \}$ belongs to $\F^+( \tau)$, that 
the family of the events $G$ which share the property stated in the claim is a $\sigma$-algebra. 
By definition, $\F^+( \tau')$ is generated by events 
of the form
$$\{ K_{ \tau'  }(x_{1},i_{1})   = k_{1}   \} \cap \{   \Psi(x_{1},i_{1},k_{2}) \in B \},
$$
for every site $x_{1}$, rule index $i_{1}$, and indices 
$k_{1}$ and $k_{2}$ such that $k_{2} \le k_{1}$, 
 and every Borel subset $B$ of $\R$.

Now for every event $G$ of this form, it is easy to check that $G \cap \{ \tau  < \tau' \}$
belongs to $\F^+( \tau)$. Claim~\ref{c01} follows.
\end{proof}

\begin{proof}[Proof of claim~\ref{c02}]
Observe that $\{ \tau  \in H_{y,\tau'} \}  $ 
 is the union over every $1\le m\le k-1$ of the events
  $$
\{ H(x-y,i,k-m )     \circ \shif_{\tau',y}   = 1 \}   \cap 
\{  K_{\tau'}(x,i) = m   \}   \cap  \{ \tau  < \tau'  \}.$$
Fix $m\le k-1$. The collection of sets $G$ in $\F^+( \Psi(x-y,i,k-m))$ such that 
the event
 $$\{ \un_{G} \circ \shif_{\tau',y}   = 1 \}   \cap \{  K_{\tau'}(x,i) = m    \} 
  \cap  \{ \tau  < \tau' \}$$
  belongs to $\F^+( \tau)$ is a $\sigma$-algebra.
On the other hand, the $\sigma$-algebra $\F^+( \tau)$ is generated by the events 
of the form
$$
\{ K_{ \tau  }(x_{1},i_{1})   = k_{1}   \} \cap \{   \Psi(x_{1},i_{1},k_{2}) \in B \},
$$
for every site $x_{1}$, rule index $i_{1}$, and integers $k_{1}$ and $k_{2}$ 
such that $k_{2} \le k_{1}$, 
 and every Borel subset $B$ of $\R$.
For every event $G$ of this form, the definition of a \ctwa\ yields the fact that the event
 $$\{ \un_{G} \circ \shif_{\tau',y}   = 1 \}   \cap \{  K_{\tau'}(x,i) = m    \} 
  \cap  \{ \tau  < \tau' \}
$$
belongs to $\F^+( \tau)$.
 Claim~\ref{c02} follows. 
\end{proof}

%
\section{Applications to nucleotide substitution models}\label{s:appli}

Most stochastic models of nucleotidic substitution processes assume that the various sites along a DNA sequence evolve independently. 
However, it is a well-known experimental fact that
the nucleotides in the immediate neighborhood of a site
can affect drastically the substitution rates at this site.
For instance, in the genomes of vertebrates, the increased rates of 
substitution of cytosine by thymine and of guanine by adenine in CpG dinucleotides
are often quite noticeable (typical ratios 10:1 when compared to the
other rates of substitution). Recently, various models that take such dependences
 into account have been proposed, 
see~\cite{ArnBurHwa, ChrHobJen, DurGal, HwaGre, JenPed, LunHei, SieHau} for instance.
Among these, the class of RN+YpR models of nucleotide substitution,  introduced by molecular biologists, and studied 
mathematically in~\cite{BerGouPia},  enjoys some
remarkable properties, such as
 the possibility to solve exactly for several quantities of interest, and the occurrence of a non-zero but finite-range dependence along the sequence. 

Since these models put restrictive conditions on substitution rates (see below) that 
may be only approximately satisfied in some actual biological situations of interest, 
 it is biologically relevant to study 
the properties of nucleotide 
substitution models that are not in the RN+YpR class but close to 
some models in this class. From a mathematical perspective, it is interesting to study what becomes of the dependencies along the sequence when small perturbations
of the RN+YpR assumptions are introduced, thus destroying 
the special mechanism leading to finite-range dependence in the RN+YpR context.  

We apply the coupling techniques described in the rest of the paper to 
a generic family of perturbations of models in this class. 
In section~\ref{ss.desc}, we describe the class of models.
In section~\ref{ss:agm}, we introduce two coupling times with ambiguities 
and we state theorem~\ref{t:sigmasac},
our main result
about these.
The proof of theorem~\ref{t:sigmasac} is in section~\ref{ss:prmr}. In section~\ref{ss:rem}, we state some remarks.
In section~\ref{ss.m}, we compute the growth parameters associated to these coupling times with ambiguities.
The computations are based on some tree decompositions of conditional distributions, 
stated in section~\ref{ss.tree}.
Finally, in section~\ref{ss.pjc}, we apply these results to the simplest non trivial example, namely the perturbed Jukes-Cantor model 
with CpG influence, thus proving the quantitative result stated as
theorem~\ref{t.jc}.

We mention that the notations in this section are sometimes
slightly at odds with those in the rest of the paper.

\subsection{Description of the models}\label{ss.desc}

\subsubsection{RN+YpR models}

Formally, these models are  interacting particle systems with state space $\A^{\Z}$, 
where 
$$
\A := \{ A,T,C,G\}
$$ 
denotes the
nucleotidic alphabet. The letters $A$ and $G$ correspond to purines, abbreviated collectively by 
$\puri$, while $C$ and $T$ correspond 
to pyrimidines, abbreviated collectively by $\pyri$.
Such a model is characterized by two sets of parameters, which describe two distinct kinds of transition mechanisms. 

The RN part of the model consists of a matrix $S:=(s_{x,y})_{x,y \in \A}$ 
of transition rates. 
The meaning of this matrix is that, when the state of a site is $x$, it is turned to $y$ at rate $s_{x,y}$, independently of the other nucleotides.
For $S$ to be an RN matrix, some identities between coefficients must hold. Specifically, $S$ must be of the following form:
$$
\begin{array}{rc}
\begin{array}{c} \\ A\\ T\\ C\\ G \end{array}
&
\!\!\!\!\begin{array}{c}
\begin{array}{cccc}A\, & \,T\, & \,C\, & \,G\end{array}
\\
\left(
\begin{array}{cccc}
- & v_{T} & v_{C} & w_{G}
\\
v_{A} & - & w_{C} & v_{G}
\\
v_{A} & w_{T} & - & v_{G}
\\
w_{A} & v_{T} & v_{C} & -
\end{array}
\right).
\end{array}
\end{array}
$$ 
The  YpR part of the model is characterized by eight transition rates $r_{x}^{y}$, indexed by all
the couples $(x,y)$ in $\A\times\A$ such that $x$
and $y$ are not both purines nor both pyrimidines.
 Thus the list of available YpR rates is
$$
r^C_{A},  \quad  r^G_{T},  \quad r^{A}_{C},  \quad r^{T}_{G} , \quad  r^C_{G},  \quad  
r^A_{T},  \quad r^{G}_{C} ,  \quad r^{T}_{A}.
$$
To describe the meaning of these rates, we introduce the notations
$$
\pyri :=\{  C, T  \},\ \I_{A}:=\pyri =:\I_{G},\ 
\puri := \{  A,G  \},\ \I_{C}:=\puri =:\I_{T},
$$
and
$$
A^{*}:=G,\quad G^{*}:=A,\quad C^{*}:=T,\quad T^{*}:=C.
$$
For every $x$ in $\puri $ and every $y$ in $\I_{x}$,  if the state of a site is $x$ and the state 
of its 
left neighbor is $y$, then the transition from $x$ to $x^{*}$ occurs at rate $r^{y}_{x}$.
Similarly,  for every $x$ in $\pyri $ and every $y$ in $\I_{x}$,  if the state of a site is $x$ 
and the 
state of its right neighbor is $y$, then the transition from $x$ to $x^{*}$ 
occurs at rate $r^{y}_{x}$.
 We refer to~\cite{BerGouPia} for 
a thorough discussion  of the properties of RN+YpR models.

In the context of this paper, we use a specification of the dynamics by transition rules which 
is not the simplest possible one, mathematically speaking, but which enjoys coupling properties 
that are
crucial in the sequel. We write a corresponding list of transition rules after another definition.

Here is the list of transition rules. Recall that every rule $\ruule$ is of the form 
$\ruule=(c,r)$ for a rate $r$ and a context 
$c=(A,\local,s)$. Accordingly, for every symbols $\alpha$ and $\beta$, 
we write $\ruule_{\alpha}^{\beta}=(c_{\alpha}^{\beta},r_{\alpha}^{\beta})$ and $c_{\alpha}^{\beta}=(A_{\alpha}^{\beta},\local_{\alpha}^{\beta},s_{\alpha}^{\beta})$.

\begin{itemize}
\item For every $x$ in $\A$,  the rule $\ruule_{x}^{U}$ is defined by
$$
A_{x}^{U}:=\emptyset,\ \local_{x}^{U}:=\emptyset,\ s_{x}^{U}:=x,\ 
r_{x}^{U}:=\min\{v_{x}, w_{x}\}.
$$
\item For every $x$ in $\A$, the rule $\ruule_{x}^{V}$ is defined by
$$
A_{x}^{V}:=\{  0   \},\ \local_{x}^{V}:= \I_{x},\ s_{x}^{V}:=x,\ 
r_{x}^{V}:= (v_{x}  - w_{x})^{+}.
$$
\item For every  $x$ in $\A$,  the rule $\ruule_{x}^{W}$ is defined by
$$
A_{x}^{W}:=\{  0   \},\ \local_{x}^{W} := \A \setminus  \I_{x},\
s_{x}^{W}:=x,\ r_{x}^{W}:=(w_{x} - v_{x})^{+}.
$$
\item For every $\{x,y\}=\pyri $ and every $z$ in $\puri $, the rule
$\ruule^{\pyri }_{xz, yz}$ is
$$
A^{\pyri }_{xz, yz}  :=\{  0, +1   \},\ s^{\pyri }_{xz, yz} :=y,\ r^{\pyri }_{xz, yz} := r^{z}_{y},
\ \local^{\pyri }_{xz, yz} :=
\{ (x,z) \}.$$
\item For every $\{x,y\}=\puri $ and every $z$ in $\pyri $, the rule
$\ruule^{\puri }_{zx, zy}$ is defined by
$$
A^{\puri }_{zx, zy}  :=\{  -1,0   \},\ s^{\puri }_{zx, zy} :=y,\ r^{\puri }_{zx, zy} := r^{z}_{y},
\ 
\local^{\puri }_{zx, zy} :=
\{ (z,x) \}.
$$
\end{itemize}

In the following, the non-degeneracy condition (ND) holds: 

\begin{quote}
(ND) 
For every nucleotide $x$ in $\A$, $w_{x}$ and $v_{x}$ are positive.
\end{quote}

\subsubsection{Perturbed RN+YpR models}

We consider perturbations equivalent to the addition to the RN rules described above, 
of a generic matrix of substitution rates, that may not satisfy the RN property, and to the addition to the YpR
rules, of generic neighbor-dependent transition rates, where the dependence is either to the right neighbor
or to the left neighbor.

Here is the list of perturbative transition rules.

\begin{itemize}
\item For every distinct $x$ and $y$ in $\A$, the rule $\ruule^{\eps}_{x,y}$ is
$$
A^{\eps}_{x,y}:=\{  0   \},\ \local^{\eps}_{x,y} := \{  x  \},\ s^{\eps}_{x,y}:=y,\ 
r^{\eps}_{x,y}:=\eps(x,y).
$$
\item For every distinct $x$ and $y$ and every $z$ in $\A$, the rule $\ruule^{\eps}_{zx,zy}$ is
$$
A^{\eps}_{zx,zy}:=\{  -1,0   \},\ \local^{\eps}_{zx,zy} := \{  (z,x)  \},\ s^{\eps}_{zx,zy}:=y,\ 
r^{\eps}_{zx,zy}:=\eps(zx,zy).
$$
\item For every distinct $x$ and $y$ and every $z$ in $\A$, the rule $\ruule^{\eps}_{xz,yz}$ is
$$
A^{\eps}_{xz,yz}:=\{  0,+1   \},\ \local^{\eps}_{xz,yz} := \{  (x,z)  \},\ s^{\eps}_{xz,yz}:=y,\ 
r^{\eps}_{xz,yz}:=\eps(xz,yz).
$$
\end{itemize}

%
\subsection{Two coupling times with ambiguities}\label{ss:agm}

\begin{notation}\label{n.sh}
For every site $x$ and every subset $K:=:\{\ruule_{i}\,;\,i\in J\}$ of $\RR$ with $J\subset\II$, let
$$\displaystyle
\Psi(x,K):=\Psi(x,J)=  \bigcup_{i\in J}\Psi(x,i),\quad\Psi(K):=\Psi(J),$$
and
$$
r(K):=\sum_{i\in J}r_{i}.
$$
By an abuse of notation, in the rest of the paper, we also use the shorthands
 $\Psi(x,\ruule_{i}):=\Psi(x,i)$ and $\Psi(\ruule_{i}):=\Psi(i)$
for every $i$ in $\II$.
\end{notation}

\begin{definition}
Let $\ZZ_{+}$, $\ZZ_{0}$, $\ZZ_{-}$, $\ZZ'_{+}$, $\ZZ'_{-}$ and $\PP$ denote subsets 
of the rule set 
$\RR$ 
such that the sets $\ZZ_{+}$, $\ZZ_{0}$, $\ZZ_{-}$ are not empty, and the sets
$$
\ZZ_{+} \cap \ZZ'_{+},\  
\ZZ_{-} \cap \ZZ'_{-},\ 
\ZZ_{0} \cap \PP,\
\ZZ_{+} \cap \PP,\
 \ZZ_{-} \cap \PP,\
\ZZ'_{+} \cap \ZZ'_{-}\cap\PP
$$
are all empty. Let  $\ZZ:=(\ZZ_{+},\ZZ_{0},\ZZ_{-},\ZZ'_{+},\ZZ'_{-})$.

Let $t_{-}$, $t_{0}$ and $t_{+}$ denote negative real numbers.
Say that a coupling event based on $\ZZ$ occurs at site $x=0$
 and at times $(t_{-}, t_{0},t_{+})$ if the following holds.

\begin{itemize}
\item $t_{-} < t_{0}$ and $t_{+} < t_{0} $.
\item $t_{-} $ belongs to $ \Psi(-1,\ZZ_{-})$, 
$t_{+} $ belongs to $ \Psi(+1,\ZZ_{+})$ and $t_{0} $ belongs to $ \Psi(0,\ZZ_{0})$.
\item The sets $\Psi(-1,\ZZ'_{-})\cap]t_{-},t_{0}[$
and $\Psi(+1,\ZZ'_{+})\cap]t_{+},t_{0}[$
are both empty.
\end{itemize}

Let $T_\ZZ$ denote the maximum of the times $\min\{t_{-},t_{+}\}$ such that a coupling event 
based on $\ZZ$ occurs at times $(t_{-},t_{0},t_{+})$.
Let $H_{\ZZ}^{\PP}$ denote the set
$$
H_{\ZZ}^{\PP} := 
\left(\Psi(-1,\PP)\cup\Psi(0,\PP)\cup\Psi(+1,\PP)\right)\cap[T_{\ZZ},0[.
$$
\end{definition}

\begin{remark}
When there exists at least one triple
$(t_{-}, t_{0},t_{+})$ which corresponds to a coupling event based on $\ZZ$, 
$T_{\ZZ}$ is indeed a maximum since the 
set 
$$
(\Psi(-1) \cup \Psi(0) \cup \Psi(+1))\cap]\min\{t_{-},t_{+}\},0[
$$ 
is finite. When there exists no triple
$(t_{-}, t_{0}, t_{+})$  which corresponds to a coupling event based on $\ZZ$,
let $T_{\ZZ}:=-\infty$.
\end{remark}

We define two examples of coupling events in the
context of perturbed RN+YpR models, that we call 
YpR sensitive and YpR insensitive.

\begin{definition}
\label{d.ce1}
YpR sensitive coupling events are based on the following choice of $\ZZ$ and $\PP$.
\begin{itemize}
\item  $\ZZ_{0}  :=  \{  \ruule_{y}^{U}\, ; \, y \in \A  \}$.
\item $\ZZ_{-} := \{\ruule^{V}_{y}\,;\,y\in\puri \} \cup \{ \ruule^{U}_{y}\, ; \,
  y \in \A \}$.
\item $\ZZ_{+} := \{\ruule_{y}^{V}\,;\,y\in\pyri\} \cup \{ \ruule_{y}^{U}\, ; \,
  y \in \A \}$.
 \item  $\ZZ'_{-} :=   \{  \ruule^{\pyri }_{xz,yz}\,; \,   \{x,y\}=\pyri , \, z \in \puri  \}  $.
  \item  $\ZZ'_{+} :=   \{  \ruule^{\puri }_{zx,zy}\,; \,   \{x,y\}=\puri , \, z \in  \pyri  \}  $.
  \item $\PP:= \{ \ruule^\eps_{zx,zy}, \,\ruule^\eps_{xz,yz}, \, \ruule^\eps_{x,y}\,;
    \, (x,y,z)\in \A^{3}\}$.
\end{itemize}
\end{definition}

\begin{definition}
\label{d.ce2}
YpR insensitive coupling events are based on the following choice of $\ZZ$ and $\PP$.

\begin{itemize}
\item  $\ZZ_{0}  :=  \{  \ruule_{y}^{U}\, ; \, y \in \A  \}$.
\item  $\ZZ_{-} := \{\ruule_{y}^{V}, \, \ruule_{y}^{U}\, ; \, y \in \puri  \}$. 
\item $\ZZ_{+} := \{\ruule_{y}^{V}, \, \ruule_{y}^{U}\, ; \, y \in \pyri  \}$.
 \item  $\ZZ'_{-} :=  
 \{\ruule_{y}^{U}, \,    \ruule_{y}^{V},\, \ruule^\eps_{x,y}, \, 
\ruule^\eps_{xz,yz} , \,\ruule^\eps_{zx,zy}\,;\,x\in\A,\,y\in\pyri,\,z\in\A \}$. 
\item  $\ZZ'_{+} :=  
\{\ruule_{y}^{U}, \,    \ruule_{y}^{V},\, \ruule^\eps_{x,y}, \, \ruule^\eps_{xz,yz} , \,\ruule^\eps_{zx,zy}\,;
\,x\in\A,\,y\in\puri,\,z\in\A \}$. 
\item $\PP:=  \{ \ruule^\eps_{x,y}, \, \ruule^\eps_{xz,yz} , \,\ruule^\eps_{zx,zy}\,;\,
(x,y,z)\in\A^{3} \}$.
\end{itemize}
\end{definition}

\begin{notation}
We write $T_{\sensi}$ and $H_{\sensi}$, respectively $T_{\insen}$ and $H_{\insen}$, 
for $T_{\ZZ}$ and $H_{\ZZ}^{\PP}$
associated to a YpR sensitive, respectively YpR insensitive, coupling event.
\end{notation}

Here is our main result about these random variables.

\begin{theorem}\label{t:sigmasac}
The random variable $(T_{\sensi}, H_{\sensi} ) $ defines a \ctwa, whose width is 
bounded by $a_{+}=a_{-}=2$, and such that $T_{\sensi}$
is exponentially integrable. The same assertions hold for $(T_{\insen},H_{\insen})$.
\end{theorem}

\begin{remark}
Due to the non-degeneracy assumption~(ND), and to standard independence properties of Poisson processes,
 $T_{\sensi}$ and $T_{\insen}$ are almost surely finite negative random variables.
Properties~(\ref{ctwa2}), (\ref{ctwa3}) and (\ref{ctwa4}) 
of definition~\ref{d:ctwa} are also easy to establish, so 
the real issue is to prove property~(\ref{ctwa5}) and the exponential integrability.
\end{remark}

\subsection{Proof of the main result}\label{ss:prmr}

This section is devoted to the proof of theorem~\ref{t:sigmasac}.

\subsubsection{Preliminary result}

\begin{notation}
Let $\varrho$ denote the application which fuses the two
purines together, and $\eta$ the application which fuses
the two pyrimidines together, that is
$$
\varrho(A):=\puri =:=\varrho(G), \quad \varrho(C):=C,\quad\varrho(T):=T,
$$
and
$$
\eta(A):=A,\quad\eta(G):=G,\quad\eta(C):=\pyri =:\eta(T).
$$
For every times $s<t$, every configuration $\xi$, let 
$$
\phi_{0}( \xi, s, t ,\Psi  )   :=  
\left( \varrho( \Phi(\xi,  s,  t, -1 )) ,  \Phi(\xi,  s, t, 0) ,    
\eta( \Phi(\xi,  s,  t, +1 ))   \right).
$$
\end{notation}

\begin{lemma}\label{l:jusquens0}
Assume that $\perf(\xi,t, H_{\sensi}) = \perf(\xi',t,H_{\sensi})$ for a time 
$t<T_{\sensi}$ and for some configurations $\xi$ and 
$\xi'$. Then, $\P$ almost surely,
$$  
\phi_{0}( \xi, t, t_{0} ,\Psi  )  = \phi_{0}( \xi', t, t_{0} ,\Psi  ).
$$
The same statement holds if one replaces $T_{\sensi}$ and $H_{\sensi}$ by $T_{\insen}$ and $H_{\insen}$,
respectively.
\end{lemma}

From now on, we assume that 
$\Omega_{1}\cap\Omega_{2} \cap \{T_{\sensi}>-\infty\}\cap\{ T_{\insen}> - \infty\}$
holds, since this event has full probability.

\subsubsection{Proof of the preliminary result
for YpR sensitive coupling events}
We study what happens if one starts at a time $t<T_{\sensi}$ from two initial configurations 
$\xi$ and $\xi'$, such that $\perf(\xi,t, H_{\sensi})=\perf(\xi',t,H_{\sensi})$.

Let $t_{1}<\cdots < t_{r}$ denote an ordering of $\Psi(-1) \cap [t_{-}, t_{0}[$. 

\begin{claim}\label{cla1}
With full probability, for every $t<T_{\sensi}$,   for every index $1\le k\le r$,
$$\varrho( \Phi(\xi,  t,  t_{k}, -1 )) =\varrho( \Phi(\xi',  t,  t_{k}, -1)).
$$
\end{claim}

If claim~\ref{cla1} holds,
$\varrho( \Phi(\xi,  t, t_{0}^{-}, -1 ))=\varrho( \Phi(\xi',  t,  t_{0}^{-}, -1))$,
and this fact implies that $\varrho( \Phi(\xi,  t, t_{0}, -1))=\varrho( \Phi(\xi',  t,  t_{0}, -1))$.

A symmetric argument shows that  $\eta(\Phi(\xi,  t, t_{0}, +1 ))=\eta( \Phi(\xi',  t,  t_{0}, +1))$.

Finally, at site $0$ and time $t_{0}$, the definition of $T_{\sensi}$ implies that there is a rule of 
the form $\ruule_y^{U}$, hence 
 $\Phi(\xi,  t, t_{0}, 0 )=\Phi(\xi',  t,  t_{0}, 0 )$ and  
lemma~\ref{l:jusquens0} holds for $T_{\sensi}$.

\begin{proof}[Proof of claim~\ref{cla1}]
Induction on $k$.
Start with $t_{1}=t_{-}$.
By definition of $T_{\sensi}$, at site $-1$ and time  $t_{-}$, there is a point 
corresponding to a rule among $\ruule^{U}_{x}$ and $\ruule^V_{y}$ for any $x$ and any purine $y$.
Each rule $\ruule_{x}^{U}$ yields a nucleotide $x$ for both initial conditions $\xi$
and $\xi'$. As regards the rules $\ruule_{A}^{V}$ and $\ruule_{G}^{V}$, 
either there is a purine at site $-1$ and time $(t_{-})^{-}$, in which case the rule is not performed, or 
there is a pyrimidine and the rule is performed. In both cases, there is a purine at site $-1$ and 
time $t_{-}$.
This proves that $\varrho( \Phi(\xi,  t, t_{-}, -1 )) =\varrho( \Phi(\xi',  t,  t_{-}, -1 ))$, 
hence the claim
holds for $k=1$.

Now we assume that the claim holds for $k-1$ with $k\le r$, hence
$\varrho( \Phi(\xi,  t,  t_k^{-}, -1 )) =\varrho( \Phi(\xi',  t,  t_k^{-}, -1 ))$ and we 
consider the effect of the rule applied at time $t_k$. Call this rule $\ruule$.
Several cases arise.

\begin{itemize}
\item
If $\ruule$ is perturbative, $\perf(\xi,t, H_{\sensi})=\perf(\xi',t,H_{\sensi})$, 
hence $\ruule$ is performed for both initial conditions $\xi$ and 
$\xi'$, or for none.
In both cases, $\varrho(\Phi(\xi,  t, t_{k}, -1)) =\varrho(\Phi(\xi',  t,  t_{k}, -1 ))$.
\item
The same reasoning holds if $\ruule$ is non-perturbative and of the form 
$\ruule_{y}^{U}$. 
\item
If $\ruule=\ruule_{y}^{V}$,
since  $\varrho( \Phi(\xi,  t,  t_k^{-}, -1 )) =\varrho( \Phi(\xi',  t,  t_k^{-}, -1))$ 
by assumption, $\Phi(\xi,  t,  t_k^{-}, -1 )$ and $\Phi(\xi',  t,  t_k^{-}, -1 )$ are both purines 
or both pyrimidines. Hence, $\ruule$
is performed for both configurations $\xi$ and $\xi'$ or for none, and
 $\varrho( \Phi(\xi,  t,  t_{k}, -1)) =\varrho( \Phi(\xi',  t,  t_{k}, -1))$. 
\item
If $\ruule=\ruule_{y}^{W}$ for a purine $y$, the application of $\ruule$ 
leaves $\varrho$ 
unchanged, since $\ruule$ can only turn
an $A$ to a $G$ or vice-versa. 
\item
If $\ruule=\ruule_{y}^{W}$ for a pyrimidine $y$,
 $\varrho(\Phi(\xi,  t,  t_k^{-}, -1 ))=\varrho(\Phi(\xi',  t,  t_k^{-}, -1 ))$, hence $\ruule$
is performed for both $\xi$ and $\xi'$, or for none.
\item
If $\ruule=\ruule^{\puri }_{zx, zy}$, the application of $\ruule$ leaves $\varrho$ 
unchanged, since $\ruule$ can only turn
an $A$ to a $G$ or vice-versa. 
 \item
Finally, the definition of $T_{\sensi}$ excludes the rules $\ruule^{\pyri }_{sz, yz}$.
\end{itemize}

This proves claim~\ref{cla1}.
\end{proof}

\subsubsection{Proof of the preliminary result for YpR insensitive coupling events}
As in the proof for $T_{\sensi}$, let $t_{1}<\cdots < t_{r}$ denote an ordering of 
$\Psi(-1) \cap [t_{-}, t_{0}[$. 

\begin{claim}\label{cla2}
With full probability, for every time $t<T_{\insen}$ and index $1\le k\le r$,
$$\varrho( \Phi(\xi,  t,  t_{k}, -1 )) =\varrho( \Phi(\xi',  t,  t_{k}, -1))=\puri .
$$
\end{claim}

If claim~\ref{cla2} holds,
$\varrho( \Phi(\xi,  t, t_{0}^{-}, -1 )) =\varrho( \Phi(\xi',  t,  t_{0}^{-}, -1))$,
and this fact implies that
$\varrho( \Phi(\xi,  t, t_{0}, -1))=\varrho( \Phi(\xi',  t,  t_{0}, -1))$.

A symmetric argument shows that  $\eta(\Phi(\xi,  t, t_{0}, +1 ))=\eta( \Phi(\xi',  t,  t_{0}, +1))$.

Finally, at site $0$ and time $t_{0}$, the definition of $T_{\insen}$ implies that there 
is a rule of 
the form $\ruule_y^{U}$, hence 
 $\Phi(\xi,  t, t_{0}, 0 )=\Phi(\xi',  t,  t_{0}, 0 )$,
and lemma~\ref{l:jusquens0} holds for $T_{\insen}$ as well.

\begin{proof}[Proof of claim~\ref{cla2}]
Induction on $k$.
Start with $t_{1}=t_{-}$.
By the definition of $T_{\insen}$, at site $-1$ and time  $t_{-}$, there is a point 
corresponding to a rule 
$\ruule_y^{U}$ or $\ruule^{V}_{y}$ with a purine $y$. 
Hence $\varrho( \Phi(\xi,  t, t_{-}, -1 ))=\varrho( \Phi(\xi',  t,  t_{-}, -1 )) $ is a purine.

Let  $k\le r$ and assume that  the claim holds for $k-1$, hence
$\varrho( \Phi(\xi,  t,  t_{k-1}, -1 ))$ and $\varrho( \Phi(\xi',  t,  t_{k-1}, -1 ))$ coincide.
Thus,
$\varrho( \Phi(\xi,  t,  t_k^{-}, -1))$ and $\varrho( \Phi(\xi',  t,  t_k^{-}, -1))$ coincide.
Consider the effect of the rule applied at time $t_k$. Call this rule $\ruule$.
Several cases arise.

\begin{itemize}
\item
Assume first that $\ruule$ is perturbative. 
If $\ruule$ is performed, $\ruule$ leads to a purine because,  by the definition of $T_{\insen}$, 
every perturbative rule applied at times $t_{1}$, \ldots, $t_{n}$, leads to a purine. 
If $\ruule$ is not performed, the induction hypothesis shows that we must as well have a purine.
As a consequence,  whether $\ruule$ is performed or not, 
$\varrho( \Phi(\xi,  t,  t_{k}, -1 ))=\varrho( \Phi(\xi',  t,  t_{k}, -1))$ is a purine.
\item
The same conclusion holds if $\ruule = \ruule_y^{U}$ or $\ruule_y^{V}$, 
because in this case,
by the  definition of $T_{\insen}$, $y$ must be a purine.
\item
If $\ruule=\ruule_y^{W}$ for a pyrimidine $y$, the induction hypothesis implies that 
$\ruule$ is not performed.
\item
If $\ruule=\ruule_y^{W}$ for a purine $y$, whether $\ruule$ is performed or not does not affect 
the value of $\varrho$.
\item
If $\ruule=\ruule^{\puri }_{zx, zy}$, performing $\ruule$ has no effect on the value of $\varrho$, 
since $\ruule$ can only turn
an $A$ to a $G$ or vice-versa. 
\item
Rules $\ruule^{\pyri }_{xz, yz}$ are not performed since
$\Phi(\xi,  t,  t_k, -1)$ 
and $\Phi(\xi',  t,  t_k, -1 )$ are both purines, by the induction hypothesis.
\end{itemize}

This proves claim~\ref{cla2}.
\end{proof}

\subsubsection{Application to theorem~\ref{t:sigmasac}}
Let $T$ denote $T_{\sensi}$ or $T_{\insen}$.
Let $t_{1}<\cdots < t_{r}$ denote an ordering of the set
$(\Psi(-1) \cup \Psi(0) \cup \Psi(+1)) \cap ]t_{0}, 0[.$

\begin{claim}\label{cla3}
With full probability, for every time $t<T$ and index $1\le k\le r$,
$$
\phi_{0}( \xi, t, t_{k} ,\Psi  )  = \phi_{0}( \xi', t, t_{k} ,\Psi  ).
$$
\end{claim}

Claim~\ref{cla3} shows that property~(\ref{ctwa5}) in definition~\ref{d:ctwa} holds. 
The  boundedness of the width is then straightforward.

\begin{proof}[Proof of claim~\ref{cla3}]
Induction on $k$. For $k=1$, this is lemma~\ref{l:jusquens0}. 
Let $k\le r$, assume that the claim holds for $t_{k-1}$ and call $\ruule$ the rule applied at time $t_{k}$. 
Let $x_k$ denote the corresponding site, hence $t_{k}$ is in $\Psi(x_{k})$. 

\begin{itemize}
\item
To begin with, if $\ruule$ is perturbative, $\ruule$ is performed for both initial conditions 
$\xi$ and $\xi'$, 
or for none, hence the claim holds for $t_{k}$. 
\item
The same is true if $\ruule=\ruule_y^{U}$ for a given $y$. 
\item
If $\ruule=\ruule_y^{V}$ for a given $y$,
$\phi_{0}( \xi, t, t_{k-1} ,\Psi  )=\phi_{0}( \xi', t, t_{k-1} ,\Psi  )$, hence, 
for $x=-1$, $0$ or $+1$, 
 $\Phi(\xi,  t,  t_k^{-}, x )$ and $\Phi(\xi',  t,  t_k^{-}, x)$ are both purines or both 
pyrimidines.
This means that $\ruule$ is performed for both $\xi$ and $\xi'$ or for none,
and the claim holds for $t_{k}$.
\item
If $\ruule=\ruule_y^W$ for a given $y$, several subcases arise.
\begin{itemize}
\item 
If $x_{k}=-1$ and $y$ is a purine, performing $\ruule$ has no effect on the value of $\varrho$, 
since $\ruule$ can only turn
an $A$ to a $G$ or vice-versa. 
\item
If $x_{k}=-1$ and $y$ is a pyrimidine,
 $\ruule$
is performed for both $\xi$ and $\xi'$ or for none, because
$\varrho(\Phi(\xi,  t,  t_k^{-}, -1 ))=\varrho(\Phi(\xi',  t,  t_k^{-}, -1 ))$.
\item
If $x_{k}=+1$, symmetric arguments hold.
\item
If $x_{k}=0$, $\Phi(\xi,  t,  t_k^{-}, -1 )=\Phi(\xi',  t,  t_k^{-}, -1 )$, hence $\ruule$
is performed for both $\xi$ and $\xi'$, or  for none.
\end{itemize}
This concludes the case when $\ruule=\ruule_y^W$ for a given $y$.
\item
If $\ruule$ is a rule $\ruule^{\puri }_{zx, zy}$ and $x_{k}=-1$, performing $\ruule$ 
has no effect on the value of $\varrho$, since it can only turn
an $A$ to a $G$ or vice-versa. 
\item
If $\ruule$ is a rule $\ruule^{\pyri }_{xz, yz}$ and $x_{k}=-1$, 
the fact that  $\varrho(\Phi(\xi,  t,  t_k^{-}, -1 ))$ and
$\varrho(\Phi(\xi',  t,  t_k^{-}, -1))$ are equal
and the fact that $\Phi(\xi,  t,  t_k^{-}, 0 )$ and 
$\Phi(\xi',  t,  t_k^{-}, 0 )$ are equal ensures
that $\ruule$
is performed for both $\xi$ and $\xi'$, or  for none.
\item
If $\ruule$ is a rule  $\ruule^{\puri }_{zx, zy}$ or a rule $\ruule^{\pyri }_{xz, yz}$, and $x_{k}=+1$, 
 similar arguments hold.
\item
If $x_{k}=0$, the facts that  $\varrho(\Phi(\xi,  t,  t_k^{-}, -1 ))$ and
$\varrho(\Phi(\xi',  t,  t_k^{-}, -1))$ are equal,
 that $\Phi(\xi,  t,  t_k^{-}, 0 )$ and $\Phi(\xi',  t,  t_k^{-}, 0 )$ are equal
for a rule of the form  $\ruule^{\puri }_{zx, zy}$, and the facts that 
  $\eta(\Phi(\xi,  t,  t_k^{-}, +1 ))$ and $\eta(\Phi(\xi',  t,  t_k^{-}, +1 ))$ are equal
and that $\Phi(\xi,  t,  t_k^{-}, 0 )$ and $\Phi(\xi',  t,  t_k^{-}, 0 )$ are equal
for a rule of the form  $\ruule^{\pyri }_{xz, yz}$ 
 ensure that $\ruule$
is performed for both $\xi$ and $\xi'$, or  for none.
\end{itemize}

This concludes the proof of claim~\ref{cla3}.
\end{proof}

\subsubsection{End of the proof of theorem~\ref{t:sigmasac}}
To conclude the proof of theorem~\ref{t:sigmasac}, one must show that $T_{\sensi}$
and $T_{\insen}$ are both exponentially integrable. The proof is the same in both cases.

We define inductively decreasing sequences of random times $(U_{n})_{n\ge0}$, 
$(U^{-}_{n})_{n\ge0}$, $(U^{+}_{n})_{n\ge0}$ and $(U^{0}_{n})_{n\ge0}$. 
Let $U_{0}:=0$.
For every $n\ge0$, let
$$
U^{-}_{n}:=\max\Psi(-1,\ZZ_{-})\cap]-\infty,U_{n}[,
\quad
U^{+}_{n}:=\max\Psi(+1,\ZZ_{+})\cap]-\infty,U_{n}[,
$$
and
$$
U^{0}_{n}:=\min\{U^{-}_{n},U^{+}_{n}\},\quad
U_{n+1}:=\max\Psi(0,\ZZ_{0})\cap]-\infty,U^{0}_{n}[.
$$
Hence $U_{n+1}<U^{0}_{n}\le U^{-}_{n},U^{+}_{n}<U_{n}$ and every $U_{n}$ is almost
surely finite.

For every $n\ge0$, consider the event 
$$
S_{n}:=\left\{\Psi(-1,\ZZ'_{-})\cap]U^{-}_{n},U_{n}[=
\Psi(+1,\ZZ'_{+})\cap]U^{+}_{n},U_{n}[=\emptyset\right\}.
$$
For every $n\ge1$, on the event $S_{n}$, 
the choice $t_{0}=U_{n}$, $t_{-}=U^{-}_{n}$, $t_{+}=U^{+}_{n}$ yields an admissible triple
$(t_{-},t_{0},t_{+})$, hence $T_{\ZZ}\ge U_{n+1}$.
Let $T:=U_{N+1}$ where $N$ is the first integer $n\ge1$ 
such that $S_{n}$ holds. We wish to show that $T$ is exponentially integrable. 

The sequence $(S_{n},U_{n}-U_{n+1})_{n\ge0}$ 
is i.i.d.\ and, for every $n\ge1$,
$$
\{N=n\}=S_{1}^{c}\cap\cdots\cap S_{n-1}^{c}\cap S_{n}.
$$
Hence, for every real number $\lambda$,
$$
\Lambda_{T}(\lambda)=\sum_{n\ge1}\E[\ee^{-\lambda U_{n+1}}\,;\,N=n]
=
\E[\ee^{-\lambda U_{1}}]\E[\ee^{-\lambda U_{1}}\,;\,S_{0}]
\sum_{n\ge1}\E[\ee^{-\lambda U_1}\,;\,S_{0}^{c}]^{n-1}.
$$
This shows that $\Lambda_{T}(\lambda)$ is finite if and only if 
$\E[\ee^{-\lambda U_1}]$ is finite and
$\E[\ee^{-\lambda U_1}\,;\,S_{0}^{c}]<1$  (but we recall that $T$ is not $T_{\ZZ}$).

By construction, $U_{1}\ge U^{-}_{0}+U^{+}_{0}+(U_{1}-U_{0}^{0})$ and these three 
random variables are independent and
exponential with parameters $r(\ZZ_{-})$, $r(\ZZ_{+})$ and $r(\ZZ_{0})$, respectively, hence 
$\E[\ee^{-\lambda U_1}]$ is finite for every $\lambda$ smaller than these three rates.
Since $U^{0}_{0}\ge U_{1}$, 
the same conclusion applies to $U^{0}_{0}$.

Furthermore, $U_{0}^{0}-U_{1}$ is independent of
$(S_{0},U_{0}^{0})$ and its distribution is exponential of parameter $r(\ZZ_{0})$.
Conditionally on $U^{+}_{0}$ and $U_{-}^{0}$, the number of points in 
the sets $\Psi(-1,\ZZ'_{-})\cap]U^{-}_{0},0[$ and 
$\Psi(+1,\ZZ'_{+})\cap]U^{+}_{0},0[$ are independent and Poisson of parameters
$-r(\ZZ'_{-})U^{-}_{0}$ and $-r(\ZZ'_{+})U^{+}_{0}$. For every $\lambda<r(\ZZ_{0})$, 
this yields
$$
\E[\ee^{-\lambda U_1}\,;\,S_{0}^{c}]=\frac{r(\ZZ_{0})}{r(\ZZ_{0})-\lambda}
\,
\E\left[(1-\ee^{r(\ZZ'_{-})U_{0}^{-}+r(\ZZ'_{+})U_{0}^{+}})
\,\ee^{-\lambda U_{0}^{0}}\right].
$$
Hence the fact that $\Lambda_{T}(\lambda)$ is finite for some 
positive $\lambda$ follows from the claim below, 
with $r_{0}:=r(\ZZ_{0})$, $W:=1-\ee^{r(\ZZ'_{-})U_{0}^{-}+r(\ZZ'_{+})U_{0}^{+}}$
and $V:=-U_{0}^{0}$.

\begin{claim}\label{cl:po}
Let $V$ and $W$ denote positive random variables and $r_{0}$ a positive real number.
Assume that $V$ is exponentially integrable and that $W<1$ almost surely.
For every real number $\lambda<r_{0}$, let
$$
F(\lambda):=\frac{r_{0}}{r_{0}-\lambda}\,
\E[W\ee^{\lambda V}].
$$
Then there exists some positive values of $\lambda$ such that
$F(\lambda)<1$.
\end{claim}

\begin{proof}[Proof of claim~\ref{cl:po}]
Expansions of the exponentials at order $1$ with respect to the parameter $\lambda$ yield 
$F(0)=\E[W]<1$, and
$
F'(0)=\E[WV]+r_{0}^{-1}\E[W].
$
Hence $F'(0)$ is finite and the proof of the claim is complete.
\end{proof}

\subsection{Remarks}\label{ss:rem}
First, in our two examples, properties~(\ref{ctwa1}), (\ref{ctwa2}), (\ref{ctwa3}) and (\ref{ctwa5}) 
in definition~\ref{d:ctwa} still hold if one removes 
the points in $\Psi(0,\PP) \cap ]T_{\ZZ},t_{0}[$ from the definition of $H^{\PP}_{\ZZ}$
(here $(t_{-},t_{0},t_{+})$ denotes a triple such that there exists a coupling event based on $\ZZ$ at 
time $(t_{-},t_{0},t_{+})$ and such that 
$\min\{t_{-},t_{+}\}$ is maximal among such coupling events). This leads to a smaller growth parameter and Laplace transform, but,
on the other hand, property~(\ref{ctwa4}) (the stopping property of $H^{\PP}_{\ZZ}$) is lost, 
and we need property~(\ref{ctwa4})
to prove effective estimates on the ambiguity process.  

Second, as regards YpR sensitive coupling events, $T_{\sensi}$ 
is defined purely in terms of non-perturbative 
rules, and, in fact, 
the proof of theorem~\ref{t:sigmasac} implies that  
$T_{\sensi}$ is an ordinary coupling time for the unperturbed dynamics.
In other words, the content of theorem~\ref{t:sigmasac} is 
that the fact of fixing the
ambiguities associated to some perturbative rules in $\Psi \cap]T_{\sensi},0[$
restores the coupling property of $T_{\sensi}$.

Third, the situation is a bit different for YpR insensitive coupling
events, whose definition involves both perturbative and
non-perturbative rules. Indeed, removing the perturbative rules from
the definitions of $\ZZ'_{-}$ and $\ZZ'_{+}$ for YpR insensitive coupling
events would make $T_{\insen}$ a coupling time for the unperturbed
dynamics. However, fixing ambiguities associated to perturbative rules
in $\Psi \cap]T_{\insen},0[$ is not enough to restore the coupling
property of $T_{\insen}$. This is the reason why, in this example, one
must introduce perturbative rules in the definition of $T_{\insen}$.

Finally, note that it is possible to use our two examples of coupling times with ambiguities to perform
perfect simulation according to the Propp-Wilson method (see~\cite{ProWil}). Indeed, by the definition of a coupling time with ambiguities, 
for any finite subset $B \subset \Z$, the random times $T_{B}^{*}$ are coalescence times that allow to sample exactly from the
projection $\pi_{B}$ of the invariant distribution of the particle system onto the sites in $B$. Moreover, from their definition, our two examples allow for an 
efficient detection of coalescence by an algorithm.

%
\subsection{Computations of growth parameters}\label{ss.m}

We recall notation~\ref{n.sh} in section~\ref{ss:agm} and
we define inductively sequences $(\lambda_{n})_{ n \ge 0}$, $(\kappa_{n})_{ n \ge 1}$
and $(\chi_{n})_{ n \ge 1}$ 
of times and a sequence $(\beta_{n})_{ n \ge 1}$ of bits, as follows.   
Let $\lambda_{0}:=0$. For every integer $n \ge 1$, let
$$
\kappa_{n}:= \sup\Psi(0,\ZZ_{0})\cap]-\infty, \lambda_{n-1}[,
$$
and
$$
\chi_{n} :=  \sup  (   \Psi(+1, \ZZ_{+} \cup \ZZ'_{+}  )
 \cup \Psi(-1, \ZZ_{-} \cup \ZZ'_{-}))\cap]-\infty, \kappa_{n}[.
$$
We define $\lambda_{n}$ and $\beta_{n}$ as follows.

\begin{itemize}
\item
If $\chi_{n}$ is in $\Psi(+1,\ZZ_{+}) $, let $\lambda_{n}:=  \sup 
\Psi(-1, \ZZ_{-} \cup \ZZ'_{-})\cap]-\infty, \kappa_{n}[ $; then, if $\lambda_{n}$ is in $ \Psi(-1,\ZZ_{-}) $,
let $\beta_{n}:=1$, otherwise let $\beta_{n}:=0$.
\item
If $\chi_{n}$ is in $\Psi(-1,\ZZ_{-}) $, let $\lambda_{n}:=  \sup  
\Psi(+1, \ZZ_{+} \cup \ZZ'_{+})
\cap]-\infty, \kappa_{n} [ $; then,  
if $\lambda_{n}$ is in $\Psi(+1,\ZZ_{+}) $,
 let $\beta_{n}:=1$, otherwise let $\beta_{n}:=0$.
\item
If $\chi_{n}$ is in $\Psi(+1,\ZZ'_{+}) $ or in $\Psi(-1,\ZZ'_{-}) $,
 let $\lambda_{n} := \chi_{n}$ and $\beta_{n}:=0$.
\end{itemize} 

For every rule $\ruule$ in $\PP$ and $n \ge 1$, let 
$$
N_n(\ruule):=\# (  \Psi(-1,\ruule) \cup  \Psi(0,\ruule) \cup  \Psi(+1,\ruule) ) \cap  
[\lambda_{n}, \lambda_{n-1}[.
 $$
By the independence properties of Poisson processes, for every rule $\ruule$ in $\PP$, 
the sequence $(\beta_{n}, N_n(\ruule))_{n \ge 1}$ is i.i.d.
Let 
$$
G(\ZZ) := \min \{   n \ge 1\,;\, \beta_{n}=1 \}.
$$ 
The distribution of $G(\ZZ)$ is geometric on $\{1,2,\ldots\}$ and non degenerate
under our non-degeneracy 
assumption that every $r_y^{U}$ is positive.

The proof of the following lemma is an easy consequence of the definitions above and is omitted.

\begin{lemma}
$\P[T_{\ZZ}=\lambda_{G(\ZZ)}]=1$. 
  \end{lemma}

As a consequence, with probability one,  for every rule $\ruule$ in $\PP$,
$$
[T_{\ZZ},0[  = \bigcup_{n=1}^{G(\ZZ)} [\lambda_{n},\lambda_{n-1}[,
$$
so that
$$
\# H_{\ZZ}^{\PP} \cap \Psi(\ruule)  =    \sum_{n=1}^{G(\ZZ)} N_n(\ruule),
$$
and
$$
\E[\# H_{\ZZ}^{\PP} \cap \Psi(\ruule)] =  \P[\beta_{1}=1]^{-1}\E[N_1(\ruule)].
$$   
The proposition below follows.

\begin{prop}\label{p.m} 
The mean of $(T_{\ZZ},H_{\ZZ}^{\PP})$ is
$$
m(T_{\ZZ},H_{\ZZ}^{\PP}) =  \P[\beta_{1}=1]^{-1}\sum_{\ruule_{i}\in\PP}  (\# A_{i}) \,
\E[N_1(\ruule_{i})].
$$
\end{prop}

Section~\ref{ss.tree} describes a tree of successive conditional distributions 
leading, 
for every rule $\ruule$ in $\PP$, 
to an explicit representation of the joint distribution of $N_1(\ruule)$ and $\beta_{1}$,
and to the computation of $\E[N_1(\ruule)]$ and $\P[\beta_{1}=1]$ 
in terms of rates.

\begin{notation}\label{n.sh2}
Introduce the shorthands $r_{0}:=r(\ZZ_{0})$,
$$
r_{1}:=r(\ZZ_{+}),\
r_{2}:=r(\ZZ'_{+}),\ r_{3}:=r(\ZZ_{+} \cup \ZZ'_{+})=r_{1}+r_{2},
$$
and
$$
r_{4}:=r(\ZZ_{-}),\ r_{5}:=r(\ZZ'_{-}),\
r_{6}:=r(\ZZ_{-} \cup \ZZ'_{-})=r_{4}+r_{5}.
$$
\end{notation}

Section~\ref{ss.tree} below leads to the following values.

\begin{itemize}
\item
$\displaystyle\P[\beta_{1}=1] =\frac{r_{1}r_{4}}{r_{3}r_{6}}$.
\item
For every $\ruule=(c,r)$ in $\PP$ and not in  $\ZZ'_{+} \cup \ZZ'_{-}$,
$$
\E[N_1(\ruule)]=3r\left(\frac1{r_{0}}+\frac1{r_{3}+r_{6}}\left(1+
\frac{r_{1}}{r_{6}}+\frac{r_{4}}{r_{3}}\right)\right).
$$
\item
For every $\ruule=(c,r)$ in $\PP$ and in $\ZZ'_{+} \cup \ZZ'_{-} $, 
$$
\E[N_1(\ruule)]=r\left(\frac3{r_{0}}+\frac1{r_{3}+r_{6}}\left(2+
3\frac{r_{1}}{r_{6}}+3\frac{r_{4}}{r_{3}}\right)\right).$$
\end{itemize}

The fact that every $\E[N_1(\ruule)]$ is bounded by the expression for the case when
$\ruule$ is not in  $\ZZ'_{+} \cup \ZZ'_{-}$ yields the following result.

\begin{prop}
One has
$\displaystyle
m(T_{\ZZ},H_{\ZZ}^{\PP}) \le3M(\ZZ)\sum_{\ruule_{i}\in\PP} r_{i}\,\# A_{i}$,
where
$$
M(\ZZ):=\frac{r_{3}r_{6}}{r_{1}r_{4}}
\left(\frac1{r_{0}}+\frac1{r_{3}+r_{6}}\left(1+
\frac{r_{1}}{r_{6}}+\frac{r_{4}}{r_{3}}\right)\right).
$$
\end{prop}

%
\subsection{Tree decompositions of conditional distributions}\label{ss.tree}
 
In this section, we describe the distributions of a family of random variables 
by placing each of them at the vertices of a tree.
The distribution of the random variable placed at vertex $x.$
is conditional to the random variables placed at vertices
which are ancestors of $x.$ 
in the tree.
We use the following labelling: $1.$ is the root of the tree, 
and the children of a vertex  
$x.$ are $x.y.$ with $y=1$, $y=2$, and so on.

To avoid cumbersome notations, we denote by $(n).$ 
a piece of ``trunk'' of length $n$, that is,  
a vertex 1.1.$\cdots$.1.\ with $n$ ones
such that no ramification starts before it. For instance,
1.1.1.=(3)., 1.1.1.1.1.1.=(5).1.\ and 1.1.1.1.1.2.=(5).2.
Finally, we recall the shorthands of notation~\ref{n.sh} in section~\ref{ss:agm}.

Three distinct situations arise: the rule 
$\ruule$ belongs to $\ZZ'_{+}$ or to $\ZZ'_{-}$ or to none of these two
sets.

\subsubsection{First case: perturbative rules not in  $\ZZ'_{+} \cup \ZZ'_{-}$}\
Assume that the rule $\ruule=(c,r)$ is in $\PP$ and not in  $\ZZ'_{+} \cup \ZZ'_{-}$.

Vertex 1. The distribution of $-\kappa_{1}$ is exponential with parameter $r_{0}$.

Vertex (2). The distribution of $\# (  \Psi(-1,\ruule) \cup  \Psi(0,\ruule) \cup  \Psi(+1,\ruule) ) \cap  
]\kappa_{1},0[$ is Poisson 
with parameter $3 r\, (-\kappa_{1})$.

Vertex (3). The distribution of $\kappa_{1}-\chi_{1}$ is exponential with parameter 
$r_{3} + r_{6}$. 

Vertex (4). The distribution of $\# (  \Psi(-1,\ruule) \cup  \Psi(0,\ruule) \cup  \Psi(+1,\ruule) ) \cap  
]\chi_{1},\kappa_{1}[$ is Poisson 
with parameter $3 r\, (\kappa_{1}-\chi_{1})$.

Vertex (5). The probability that  $\chi_{1}$ belongs to $\Psi(+1,\ZZ_{+} \cup \ZZ'_{+})$ 
is $r_{3}/(r_{3} + r_{6})$, hence
the probability that  $\chi_{1}$ belongs to $\Psi(-1,\ZZ_{-} \cup \ZZ'_{-})$ is 
$r_{6}/(r_{3} + r_{6})$.

Vertex (5).1. If  $\chi_{1}$ belongs to $\Psi(+1,\ZZ_{+} \cup \ZZ'_{+})$, then the probability that  
$\chi_{1}$ belongs to $\Psi(+1,\ZZ_{+})$ 
 is $r_{1} /r_{3}$, hence the probability that  $\chi_{1}$ belongs to $\Psi(+1,\ZZ'_{+})$ 
is $r_{2} /r_{3}$.
 
Vertex (5).1.1. If $\chi_{1}$ belongs to $\Psi(+1,\ZZ'_{+})$, then $\beta_{1}=0$ and $\lambda_{1}=\chi_{1}$.

Vertex (5).1.2. If $\chi_{1}$ belongs to $\Psi(+1,\ZZ_{+})$, then the distribution of $\chi_{1}-\lambda_{1}$ 
is exponential with parameter 
$r_{6}$.

Vertex (5).1.2.1. The distribution of $\# (  \Psi(-1,\ruule) \cup  \Psi(0,\ruule) \cup  \Psi(+1,\ruule) ) \cap  
]\lambda_{1},\chi_{1}[$ is Poisson 
with parameter $3 r\, (\chi_{1} -\lambda_{1} )$.

Vertex (5).1.2.1.1. The probability that $\beta_{1}=1$ is
$r_4/r_{6}$, hence the probability that $\beta_{1}=0$ is $r_{5}/r_{6}$.

Vertex (5).2. If   $\chi_{1}$ belongs to $\Psi(-1,\ZZ_{-} \cup \ZZ'_{-})$, then the probability that  
$\chi_{1}$ belongs to $\Psi(-1,\ZZ_{-})$ is
 $r_{4} /r_{6}$, hence the probability that  $\chi_{1}$ belongs to $\Psi(-1,\ZZ'_{-})$ 
 is $r_{5}/ r_{6}$.
 
Vertex (5).2.1. If $\chi_{1}$ belongs to $ \Psi(-1,\ZZ'_{-})$, then $\beta_{1}=0$ and 
$\lambda_{1}=\chi_{1}$.

Vertex (5).2.2. If  $\chi_{1}$ belongs to $ \Psi(-1,\ZZ_{-})$, then the distribution of 
$\chi_{1}-\lambda_{1}$ is exponential with parameter 
$r_{3}$.

Vertex (5).2.2.1. The distribution of $\# (  \Psi(-1,\ruule) \cup  \Psi(0,\ruule) \cup  \Psi(+1,\ruule) ) \cap  
]\lambda_{1},\chi_{1}[$ is Poisson 
with parameter $3r\,(\chi_{1} -\lambda_{1})$. 

Vertex (5).2.2.1.1. The probability that $\beta_{1}=1$ is
$r_{1}/r_{3}$, hence the probability that $\beta_{1}=0$ is
$r_{2}/r_{3}$. 

\subsubsection{Second case: perturbative rules in $\ZZ'_{+}$}\
Assume that the rule $\ruule=(c,r)$ is in $\PP\cap\ZZ'_{+}$.

Vertex 1. The distribution of $-\kappa_{1}$ is exponential with parameter $r_0$.

Vertex (2). The distribution of $\# (  \Psi(-1,\ruule) \cup  \Psi(0,\ruule) \cup  \Psi(+1,\ruule) ) \cap  
]\kappa_{1},0[$ is Poisson 
with parameter $3 r\,(-\kappa_{1})$.

Vertex (3). The distribution of $\kappa_{1}-\chi_{1}$ is exponential with parameter 
$r_3 + r_6$. 

Vertex (4). The distribution of $\# (  \Psi(-1,\ruule) \cup  \Psi(0,\ruule) \cup  \Psi(+1,\ruule) ) \cap  
]\chi_{1},\kappa_{1}[$ is Poisson 
with parameter $r\,(\kappa_{1}-\chi_{1})$.

Vertex (5). The probability that  $\chi_{1}$ belongs to $\Psi(+1,\ZZ_{+} \cup \ZZ'_{+})$ is 
$r_{3}/(r_{3} + r_{6})$, hence
the probability that  $\chi_{1}$ belongs to $\Psi(-1,\ZZ_{-} \cup \ZZ'_{-})$ is 
$r_{6}/(r_{3} + r_{6})$.

Vertex (5).1. If  $\chi_{1}$ belongs to $\Psi(+1,\ZZ_{+} \cup \ZZ'_{+})$, then the probability that  
$\chi_{1}$ belongs to $\Psi(+1,\ZZ_{+})$ 
is $r_{1}/r_{3}$, hence the probability that  $\chi_{1}$ belongs to $\Psi(+1,\ZZ'_{+})$ 
is $r_2/r_{3}$.

Vertex (5).1.1. If $\chi_{1}$ belongs to $\Psi(+1,\ZZ'_{+})$, then $\beta_{1}=0$ and $\lambda_{1}=\chi_{1}$,
 and the probability that
$\chi_{1}$ belongs to $\Psi(+1,\ruule)$ is $r/r_{2}$.

Vertex (5).1.2. If $\chi_{1}$ belongs to $\Psi(+1,\ZZ_{+})$, then the distribution of $\chi_{1}-\lambda_{1}$ 
is exponential with parameter $r_{6}$.

Vertex (5).1.2.1. The distribution of $\# (  \Psi(-1,\ruule) \cup  \Psi(0,\ruule) \cup  \Psi(+1,\ruule) ) 
\cap  ]\lambda_{1},\chi_{1}[$ is Poisson with parameter $3r\,(\chi_{1} -\lambda_{1})$.

Vertex (5).1.2.1.1. The probability that $\beta_{1}=1$ is
$r_4/ r_{6}$, hence the probability that  $\beta_{1}=0$ is $r_{5}  /r_{6}$.

Vertex (5).2. If   $\chi_{1}$ belongs to $\Psi(-1,\ZZ_{-} \cup \ZZ'_{-})$, the probability that  
$\chi_{1}$ belongs to $\Psi(-1,\ZZ_{-})$ is
 $r_{4} /r_{6}$, hence the probability that  $\chi_{1}$ belongs to $\Psi(-1,\ZZ'_{-})$ 
is $r_5/r_{6}$.
 
Vertex (5).2.1. If $\chi_{1}$ belongs to $ \Psi(-1,\ZZ'_{-})$, then $\beta_{1}=0$ and 
$\lambda_{1}=\chi_{1}$.

Vertex (5).2.2. If  $\chi_{1}$ belongs to $ \Psi(-1,\ZZ_{-})$, then the distribution of 
$\chi_{1}-\lambda_{1}$ is exponential with parameter $r_3$.

Vertex (5).2.2.1. The distribution of $\# (  \Psi(-1,\ruule) \cup  \Psi(0,\ruule) \cup  \Psi(+1,\ruule) ) 
\cap  ]\lambda_{1},\chi_{1}[$ is Poisson 
with parameter $2r\,(\chi_{1} -\lambda_{1} )$. 

Vertex (5).2.2.1.1. The probability that $\beta_{1}=1$ is
$  r_{1}/ r_3 $, hence the probability that $\beta_{1}=0$ is
$  r_{2}/r_{3}$. If $\beta_{1}=0$, the probability that   $\chi_{1}$ belongs to $\Psi(+1,\ruule)$ is $r/r_{2}$.

\subsubsection{Third case: perturbative rules in $\ZZ'_{-}$} 

One can deduce this case from the second case:   
 a  similar algorithm holds, obtained through the transformations 
$\ZZ_{-} \leftrightarrow \ZZ_{+}$ and $\ZZ'_{-} \leftrightarrow \ZZ'_{+}$ 
in the steps of the algorithm for rules in $\ZZ'_{+}$.

%
\subsection{Example: perturbed Jukes-Cantor model with CpG influence}\label{ss.pjc}

We now apply this upper bound to perturbations of Jukes-Cantor model of evolution with 
influence of the dinucleotide CpG. Namely, we 
assume that CpG mutates to CpA and to TpG, both at rates 
$\tautaux$, that
every nucleotide $x$ mutates to $y\ne x$ at rate $1+\eps(x,y)$ with $\eps(x,y)\ge0$.
Let 
$|\eps|$ denote the sum over every $x$ and $y$ of the perturbations $\eps(x,y)$.

As regards sensitive coupling events, $r_{0}=r_{1}=r_{4}=4$ and $r_{2}=r_{5}=\tautaux$.
The mean of the coupling mechanism based on $(T_{\sensi},H_{\sensi})$ is bounded by
$\frac3{64}|\eps|(40+10\tautaux+\tautaux^{2})$, 
hence this coupling mechanism is subcritical
as soon as
$|\eps|<\eps_{\sensi}(\tautaux)$, with
$$\eps_{\sensi}(\tautaux):=
\frac{64}{3(40+10\tautaux+\tautaux^{2})}.
$$
Hence, $\eps_{\sensi}(0)=\frac8{15}$, $\eps_{\sensi}(2)=\frac13$, 
$\eps_{\sensi}(10)=\frac4{45}$, and
$\eps_{\sensi}(\tautaux)\to0$ when $\tautaux\to+\infty$.

As regards insensitive coupling events, $r_{0}=4$, $r_{1}=r_{4}=2$, 
$r_{2}$ is the sum of the modifications
$\eps(x,y)$ for $y$ in $R$, and
$r_{5}$ is the sum of the modifications
$\eps(x,y)$ for $y$ in $Y$.
The mean of the coupling mechanism based on $(T_{\insen},H_{\insen})$ is bounded by
$\frac3{64}|\eps|(64+12|\eps|+|\eps|^{2})$, hence this coupling mechanism is subcritical
as soon as
$|\eps|<\eps_{\insen}$, where $\eps_{\insen}$ is the unique positive root of
$$
3\eps_{\insen}(64+12\eps_{\insen}+\eps_{\insen}^{2})=64.
$$
For every value of $\tautaux$, this insensitive coupling mechanism is subcritical
as soon as $|\eps|\le\frac3{10}$.  

\begin{theorem}\label{t.jc}
Every perturbation by rates $\eps$ of the 
Jukes-Cantor model with CpG influence of magnitude $\tautaux$ is ergodic as soon as
$|\eps|<\max\{\eps_{\insen},\eps_{\sensi}(\tautaux)\}$. 
Furthermore, every finite marginal converges exponentially fast to the corresponding
finite marginal of the 
stationary distribution, and the correlations of the stationary distribution decay exponentially fast.
\end{theorem}

%

\begin{thebibliography}{10}

\bibitem{AltAvrNun}
Eitan Altman, Konstantin~E. Avrachenkov, and Rudesindo N{\'u}{\~n}ez-Queija.
\newblock Perturbation analysis for denumerable {M}arkov chains with
  application to queueing models.
\newblock {\em Advances in Applied Probability}, 36(3):839--853, 2004.

\bibitem{ArnBurHwa}
Peter Arndt, Chris Burge, and Terence Hwa.
\newblock {DNA} sequence evolution with neighbour-dependent mutation.
\newblock {\em Journal of Computational Biology}, 10:313--322, 2003.

\bibitem{BerGouPia}
Jean B\'erard, Jean-Baptiste Gou\'er\'e, and Didier Piau.
\newblock Solvable models of neighbor-dependent nucleotide substitution
  processes.
\newblock {\em Mathematical Biosciences}, To appear.

\bibitem{ChrHobJen}
Ole Christensen, Asger Hobolth, and Jens~Ledet Jensen.
\newblock Pseudo-likelihood analysis of codon substitution models with neighbor
  dependent rates.
\newblock {\em Journal of Computational Biology}, 12:1166--1182, 2005.

\bibitem{DurGal}
Laurent Duret and Nicolas Galtier.
\newblock The covariation between {T}p{A} deficiency, {C}p{G} deficiency, and
  {G}+{C} content of human isochores is due to a mathematical artifact.
\newblock {\em Molecular Biology and Evolution}, 27:1620--1625, 2000.

\bibitem{GlyMey}
Peter~W. Glynn and Sean~P. Meyn.
\newblock A {L}iapounov bound for solutions of the {P}oisson equation.
\newblock {\em Ann. Probab.}, 24(2):916--931, 1996.

\bibitem{HwaGre}
Dick~G. Hwang and Phil Green.
\newblock Bayesian {M}arkov chain {M}onte {C}arlo sequence analysis reveals
  varying neutral substitution patterns in mammalian evolution.
\newblock {\em Proceedings of the National Academy of Sciences USA},
  101:13994--14001, 2004.

\bibitem{JenPed}
Jens~Ledet Jensen and Anne-Mette~Krabbe Pedersen.
\newblock Probabilistic models of {DNA} sequence evolution with context
  dependent rates of substitution.
\newblock {\em Advances in Applied Probability}, 32(2):499--517, 2000.

\bibitem{KomOll}
Tomasz Komorowski and Stefano Olla.
\newblock On mobility and {E}instein relation for tracers in time-mixing random
  environments.
\newblock {\em Journal of Statistical Physics}, 118(3-4):407--435, 2005.

\bibitem{Lig}
Thomas~M. Liggett.
\newblock {\em Interacting particle systems}, volume 276 of {\em Grundlehren
  der Mathematischen Wissenschaften}.
\newblock Springer-Verlag, New York, 1985.

\bibitem{Lig2}
Thomas~M. Liggett.
\newblock {\em Stochastic interacting systems: contact, voter and exclusion
  processes}, volume 324 of {\em Grundlehren der Mathematischen
  Wissenschaften}.
\newblock Springer-Verlag, Berlin, 1999.

\bibitem{LunHei}
Gerton~A. Lunter and Jotun Hein.
\newblock A nucleotide substitution model with nearest-neighbour interactions.
\newblock {\em Bioinformatics}, 20:i216--i223, 2004.

\bibitem{MaeNet}
Christian Maes and Karel Neto{\v{c}}n{\'y}.
\newblock Spacetime expansions for weakly coupled interacting particle systems.
\newblock {\em Journal of Physics A, Mathematical and General},
  35(13):3053--3077, 2002.

\bibitem{ProWil}
James~Gary Propp and David~Bruce Wilson.
\newblock Exact sampling with coupled {M}arkov chains and applications to
  statistical mechanics.
\newblock {\em Random Structures Algorithms}, 9(1-2):223--252, 1996.

\bibitem{RobRosSch}
Gareth~O. Roberts, Jeffrey~S. Rosenthal, and Peter~O. Schwartz.
\newblock Convergence properties of perturbed {M}arkov chains.
\newblock {\em J. Appl. Probab.}, 35(1):1--11, 1998.

\bibitem{ShaStu}
Tony Shardlow and Andrew~M. Stuart.
\newblock A perturbation theory for ergodic {M}arkov chains and application to
  numerical approximations.
\newblock {\em SIAM Journal on Numerical Analysis}, 37(4):1120--1137
  (electronic), 2000.

\bibitem{SieHau}
Adam Siepel and David Haussler.
\newblock Phylogenetic estimation of context-dependent substitution rates by
  maximum likelihood.
\newblock {\em Molecular Biology and Evolution}, 21:468--488, 2004.

\end{thebibliography}

\end{document}